%% file: stretch_new.tex
\documentclass[reqno]{amsart}
\usepackage{epsfig,color}

\newcommand{\R}{\mathbb{R}}
\newcommand{\Rplus}{\mathbb{R}_+}
\newcommand{\N}{\mathbb{N}}

\newcommand{\union}{\cup}

\newcommand\intersection{\cap}
\newcommand\closure{\operatorname{cl}}
\newcommand\after{\circ}

\newcommand\measuredlams{\mathcal{ML}}
\newcommand\projmeasuredlams{{\mathcal{PML}}}
\newcommand\measuredfoliations{\mathcal{MF}}
\newcommand\sing{\operatorname{Sing}}

\newcommand\modulargroup{\operatorname{Mod}_S}
\newcommand\teichmullerspace{\mathcal{T}(S)}
\newcommand\teichmullerspc{\mathcal{T}}
\newcommand\thurstoncompact{\mathcal{T}^T(S)}
\newcommand\hlength{\ell}
\newcommand\unitlams{\mathcal{P}}
\newcommand\bonint{i}
\newcommand\leb{\mu}

\newcommand\iso{f}
\newcommand\mapclass{h}
\newcommand\distfn{\psi}
\newcommand\horocyclic{\phi}

\newcommand\curvecomp{\mathcal{C}(S)}
\newcommand\curvecmp{\mathcal{C}}
\newcommand\curve{\alpha}
\newcommand\scurv{\beta}
\newcommand\curves{\mathcal{S}}

\newcommand\slam{\nu}
\newcommand\blam{\mu}

\newcommand\lengthfunc{\mathcal{L}}
\newcommand\stretchdist{L}
\newcommand\lfactor{Q}
\newcommand\stretchline{\Gamma}
\newcommand\dist{d}
\newcommand\symdist{d_{\text{sym}}}
\newcommand\maxset{R}
\newcommand\geo{\gamma}

\newcommand\isometric{\cong}

\newtheorem*{theorem*}{Theorem}
\newtheorem*{proposition*}{Proposition}

\newtheorem*{thmiso}{Theorem~\ref{thm:isometries}}
\newtheorem*{thmhoro}{Theorem~\ref{thm:horothurston}}
\newtheorem*{thmdistinct}{Theorem~\ref{thm:distinct}}

\newenvironment{remark}[1][Remark.]{\begin{trivlist}
\item[\hskip \labelsep {\bfseries #1}]}{\end{trivlist}}

\newtheorem{prop}{Proposition}[section]
\newtheorem{proposition}[prop]{Proposition}
\newtheorem{corollary}[prop]{Corollary}
\newtheorem{lemma}[prop]{Lemma}

\newtheorem{assumption}{Assumption}

\newtheorem{theorem}[prop]{Theorem}
\theoremstyle{definition}

\begin{document}

\title[Horoboundary and Isometry group of Thurston's Metric]
{The Horoboundary and Isometry group of Thurston's Lipschitz Metric}
\author{Cormac Walsh}
\address{INRIA Saclay \& Centre de Math\'ematiques Appliqu\'es,
Ecole Polytechnique, 91128 Palaiseau, France}
\email{cormac.walsh@inria.fr}


\begin{abstract}
We show that the horofunction boundary of Teichm\"uller space with Thurston's
Lipschitz metric is the same as the Thurston boundary.
We use this to determine the isometry group of the Lipschitz metric,
apart from in some exceptional cases.
We also show that the Teichm\"uller spaces of different surfaces, when
endowed with this metric, are not isometric, again with some possible
exceptions of low genus.
\end{abstract}

\maketitle

\section{Introduction}
\label{sec:intro}

Consider a connected oriented surface $S$ of negative Euler characteristic.
A hyperbolic metric on $S$ is a Riemannian metric of constant curvature $-1$.
The Teichm\"uller space $\teichmullerspace$ of $S$ is the space of complete
finite-area hyperbolic metrics on $S$ up to isotopy.

In~\cite{thurston_minimal}, Thurston defined an asymmetric metric on
Teichm\"uller space:
\begin{align*}
\stretchdist(x,y) := \log \inf_{\phi\approx\text{Id}} \sup_{p\neq q}
\frac{d_y(\phi(p),\phi(q))}{d_x(p,q)},
\qquad\text{for $x,y\in\teichmullerspace$}.
\end{align*}
In other words, the distance from $x$ to $y$ is the logarithm of the smallest
Lipschitz constant over all homeomorphisms from $x$ to $y$ that are isotopic
to the identity.
Thurston showed that this is indeed a metric, although a non-symmetric one.
In the same paper, he showed that this distance can be written
\begin{align*}
\stretchdist(x,y)
   = \log \sup_{\curve\in\curves} \frac{\hlength_y(\alpha)}{\hlength_x(\alpha)},
\end{align*}
where $\curves$ is the set of isotopy classes of non-peripheral simple
closed curves on $S$, and $\hlength_x(\alpha)$ denotes the shortest length
in the metric $x$ of a curve isotopic to $\alpha$.

Thurston's Lipschitz metric\index{Thurston's Lipschitz metric}\index{Thurston's metric}\index{Lipschitz metric}
has not been as intensively studied as the
Teichm\"uller metric or the Weil--Petersson metric.
The literature includes~\cite{papadopoulos_extension, theret_thesis,
papadopoulos_theret_topology, theret_negative, papadopoulos_theret_polyhedral,
theret_elementary,choi_rafi_comparison, theret_divergence};
most of it is very recent.

In this chapter, we determine the horofunction boundary of the Lipschitz metric.
The horofunction boundary of a metric space was introduced by
Gromov in~\cite{gromov:hyperbolicmanifolds}, and
has applications in studying isometry groups~\cite{walsh_lemmens_polyhedral},
random walks~\cite{karlsson_ledrappier_laws},
quantum metric spaces~\cite{rieffel_group},
and is the right setting for Patterson--Sullivan
measures~\cite{burger_mozes_commensurators}.

We show that the compactification of Teichm\"uller space by horofunctions
is isomorphic to the Thurston compactification, and we give an explicit
expression for the horofunctions. Denote by $\measuredlams$ the space
of measured laminations on $S$
and by $\projmeasuredlams$ its projectivization.

\begin{thmhoro}
A sequence $x_n$ in $\teichmullerspace$ converges in the Thurston
compactification\index{Thurston compactification}\index{compactification!Thurston}
if and only if it converges in the horofunction compactification.
If the limit in the Thurston compactification is the projective class
$[\mu]\in\projmeasuredlams$, then the limiting horofunction is
\begin{align*}
\Psi_\mu(x) =
   \log\left(\sup_{\eta\in\measuredlams}
         \frac{i(\mu,\eta)}{\hlength_x(\eta)}
   \middle/
   \sup_{\eta\in\measuredlams}
         \frac{i(\mu,\eta)}{\hlength_b(\eta)}
   \right).
\end{align*}
\end{thmhoro}
Here, $b$ is a base-point in $\teichmullerspace$, and
$i(\cdot,\cdot)$ denotes the geometric intersection number.
Recall that the latter is defined for pairs of curve classes
$(\alpha,\beta)\in\curves\times\curves$ to be the minimum number of
transverse intersection points of curves $\alpha'$ and $\beta'$ with
$\alpha'\in\alpha$ and $\beta'\in\beta$. This minimum is realised if
$\alpha'$ and $\beta'$ are closed geodesics. The geometric intersection number
extends to a continuous symmetric function on
$\measuredlams\times\measuredlams$.

It is known that geodesics always converge to a point in the horofunction
boundary; see Section~\ref{horoboundary}.
Hence, an immediate consequence of the above theorem is the following.
\begin{corollary}
Every geodesic of Thurston's Lipschitz metric converges in the forward direction
to a point in the Thurston boundary.\index{Thurston boundary}\index{boundary!Thurston}
\end{corollary}
This generalises a result of Papadopoulos~\cite{papadopoulos_extension},
which states that every member of a special class of geodesics,
the \emph{stretch lines}, converges in the forward direction to a point
in the Thurston boundary.

The action of the isometry group of a metric space extends continuously
to an action by homeomorphisms on the horofunction boundary.
Thus, the horofunction boundary is useful for studying groups
of isometries of metric spaces.
One of the tools it provides is the \emph{detour cost}, which is a kind of
metric on the boundary. We calculate this in Section~\ref{sec:thurston_detour}.

Denote by $\modulargroup$ the extended mapping class
group\index{extended mapping class group}\index{modular group}
of $S$, that is, the group of isotopy classes of homeomorphisms of $S$.
It is easy to see that $\modulargroup$ acts by isometries on
$\teichmullerspace$ with the Lipschitz metric.
We use the detour cost to prove the following.
\begin{thmiso}
If $S$ is not a sphere with four or fewer punctures,
nor a torus with two or fewer punctures, then every isometry of
$\teichmullerspace$ with Thurston's Lipschitz metric is an element
of the extended mapping class group $\modulargroup$.
\end{thmiso}
This answers a question in~\cite[\S4]{papadopoulos_theret_handbook}.

It is well known that the subgroup of elements of $\modulargroup$ acting
trivially on $\teichmullerspace$ is of order two if $S$ is the closed surface
of genus two,
and is just the identity element in the other cases considered here.

Theorem~\ref{thm:isometries} is an analogue of Royden's theorem\index{Royden's theorem}
concerning the Teichm\"uller metric, which was proved by
Royden~\cite{royden_isometries} in the case of compact surfaces and analytic
automorphisms of $\teichmullerspace$, and extended to the general case by
Earle and Kra~\cite{earle_kra_isometries}.
Our proof is inspired by Ivanov's proof of Royden's theorem,
which was global and geometric in nature, as opposed to the original,
which was local and analytic.

The following theorem shows that distinct surfaces give rise to distinct
Teichm\"uller spaces, except possibly in certain cases.
Denote by $S_{g,n}$ a surface of genus $g$ with $n$ punctures.
\begin{thmdistinct}
Let $S_{g,n}$ and $S_{g',n'}$ be surfaces of negative Euler characteristic.
Assume $\{(g,n), (g',n')\}$ is different from each of the three sets
\begin{align*}
\{(1,1), (0,4)\}, \quad \{(1,2), (0,5)\}, \quad\text{and}\quad
\{(2,0), (0,6)\}.
\end{align*}
If $(g,n)$ and $(g',n')$ are distinct, then the Teichm\"uller spaces
$\teichmullerspc(S_{g,n})$ and $\teichmullerspc(S_{g',n'})$ with their
respective Lipschitz metrics are not isometric.
\end{thmdistinct}
This is an analogue of a theorem of Patterson~\cite{patterson_distinct}.
\index{Patterson's Theorem}
In the case of the Teichm\"uller metric it is known that one has the following
three isometric equivalences:
$\teichmullerspc(S_{1,1})\isometric \teichmullerspc(S_{0,4})$, $\teichmullerspc(S_{1,2})\isometric \teichmullerspc(S_{0,5})$, and
$\teichmullerspc(S_{2,0})\isometric \teichmullerspc(S_{0,6})$.
It would be interesting to know if these equivalences still hold when
one takes instead the Lipschitz metric.

It would also be interesting to work out the horofunction boundary of
the\index{reversed Lipschitz metric}\index{Lipschitz metric!reversed}
reversed Lipschitz metric,
that is, the metric $L^*(x,y):=L(y,x)$.
Since $L$ is not symmetric, $L^*$ differs from $L$, and one would expect
their horofunction boundaries to also differ.

The contents of this chapter are as follows. In the next section, we recall the
horofunction boundary of a metric space. In section~\ref{sec:thurston_horo},
we show that the horofunction boundary of the Lipschitz metric is the Thurston
boundary. In section~\ref{sec:busemann_points}, we recall the definition
of stretch line and prove that all horofunctions of the Lipschitz metric
are Busemann points. In section~\ref{sec:detour}, we recall some results
about the detour cost, which we calculate for the Lipschitz metric
in section~\ref{sec:thurston_detour}. Section~\ref{sec:isometries} is
devoted to the proofs of Theorems~\ref{thm:isometries} and~\ref{thm:distinct}.

\paragraph{Acknowledgments:}
We thank Weixu Su for some comments on an early draft and Micka\"el Crampon
for a detailed reading.

\section{The horofunction boundary}
\label{horoboundary}

Let $(X,\dist)$ be a possibly non-symmetric metric space,
\index{non-symmetric metric space}
in other words, $\dist$ has all the properties of a metric except that
it is not necessarily symmetric.

We endow $X$ with the topology induced by the symmetrised metric
\index{symmetrised metric}
$\symdist(x,y):=\dist(x,y)+\dist(y,x)$.
Note that for Thurston's Lipschitz metric, this topology is just the usual one
on $\teichmullerspace$; see~\cite{papadopoulos_theret_topology}.

The horofunction boundary of $(X,\dist)$ is defined as follows.\index{horofunction boundary}\index{boundary!horofunction}
One assigns to each point $z\in X$ the function
$\distfn_z:X\to \R$,
\begin{equation*}
\distfn_z(x) := \dist(x,z)-\dist(b,z),
\end{equation*}
where $b$ is some base-point.
Consider the map $\distfn:X\to C(X),\, z\mapsto \distfn_z$ from
$X$ into $C(X)$, the space of continuous real-valued functions
on $X$ endowed with the topology of uniform convergence on bounded sets
of $\symdist$.
\begin{proposition}[\cite{ballmann_lectures}]
\label{prop:injective_continous}
The map $\distfn$ is injective and continuous.
\end{proposition}
\begin{proof}
The triangle inequality implies that
$\distfn_x(\cdot) - \distfn_y(\cdot) \le \dist(y,x) + \dist(x,y)$,
for all $x$ and $y$ in $X$. The continuity of $\distfn$ follows.

Let $x$ and $y$ be distinct points in $X$, and relabel them such that
$\dist(b,x)\ge\dist(b,y)$. We have
\begin{align*}
\distfn_y(x) - \distfn_x(x)
   &= \dist(x,y) - \dist(b,y) - \dist(x,x) + \dist(b,x) \\
   &\ge \dist(x,y),
\end{align*}
which shows that $\distfn_x$ and $\distfn_y$ are distinct.
\end{proof}
The horofunction boundary is defined to be
\begin{align*}
X(\infty):=\closure\{\distfn_z\mid z\in X\}\backslash\{\distfn_z\mid z\in X\},
\end{align*}
where $\closure{}$ denotes the closure of a set.
The elements of $X(\infty)$ are called horofunctions.\index{horofunction}
This definition first appeared, for the case of symmetric metrics,
in~\cite{gromov:hyperbolicmanifolds}.
For more information, see~\cite{ballmann_lectures}, \cite{rieffel_group},
and~\cite{AGW-m}.

One may check that if one changes to an alternative base-point $b'$, then the
new function assigned to a point $z$ is related to the old by
$\distfn'_z(\cdot) = \distfn_z(\cdot) - \distfn_z(b')$.
It follows that the horofunction boundary obtained using $b'$ is
homeomorphic to that obtained using $b$, and the horofunctions are related by
$\xi'(\cdot) = \xi(\cdot) - \xi(b')$.

Note that, if the metric $\symdist$ is proper,\index{proper metric space}
meaning that closed balls are compact, then uniform convergence on bounded sets
is equivalent to uniform convergence on compact sets.

The functions $\{\distfn_z\mid z\in X\}$ satisfy
$\distfn_z(x) \le \dist(x,y) + \distfn_z(y)$ for all $x$ and $y$.
Hence, for all horofunctions $\eta$,
\begin{align}
\label{eqn:ideal_tri_ineq}
\eta(x) \le d(x,y) + \eta(y),
\qquad\text{for all $x$ and $y$ in $X$}.
\end{align}
It follows that all elements of $\closure\{\distfn_z\mid z\in X\}$
are $1$-Lipschitz with respect to the metric $\symdist$.
We conclude that, for functions in this set, uniform convergence on compact
sets is equivalent to pointwise convergence.

Moreover, if $\symdist$ is proper, then the set
$\closure\{\distfn_z\mid z\in X\}$ is compact by the
Ascoli--Arzel\`a Theorem,\index{Ascoli-Arzela Theorem@Ascoli--Arzel\`a Theorem}
and we call it the \emph{horofunction compactification}.

The following assumptions will be useful.
They hold for the Lipschitz metric; see~\cite{papadopoulos_theret_topology}.
\begin{assumption}
\label{ass:proper}
The metric $\symdist$ is proper.
\end{assumption}
A geodesic\index{geodesic}
in a possibly non-symmetric metric space $(X,\dist)$ is a map
$\gamma$ from a closed interval of $\R$ to $X$ such that
\begin{align*}
\dist(\gamma(s),\gamma(t)) = t-s,
\end{align*}
for all $s$ and $t$ in the domain, with $s<t$.
\begin{assumption}
\label{ass:geodesic}
Between any pair of points in $X$, there exists a geodesic with respect to $d$.
\end{assumption}
\begin{assumption}
For any point $x$ and sequence $x_n$ in $X$, we have $\dist(x_n,x)\to 0$
if and only if $\dist(x,x_n) \to 0$.
\label{ass:topology}
\end{assumption}

\begin{proposition}[\cite{ballmann_lectures}]
\label{prop:embedding}
Assume~\ref{ass:proper}, \ref{ass:geodesic}, and~\ref{ass:topology} hold.
Then, $\distfn$ is an embedding of $X$ into $C(X)$, in other words, is a
homeomorphism from $X$ to its image.
\end{proposition}
\begin{proof}
That $\distfn$ is injective and continuous was proved in
Proposition~\ref{prop:injective_continous}.

Let $z_n$ be a sequence in $X$ escaping to infinity, that is, 
eventually leaving and never returning to every compact set.
We wish to show that no subsequence of
$\distfn_{z_n}$ converges to a function $\distfn_{y}$ with $y\in X$.
Without loss of generality, assume that $\distfn_{z_n}$ converges
to $\xi\in\closure\{\distfn_z\mid z\in X\}$.

Since $\symdist$ is proper, $\symdist(y,z_n)$ must converge to
infinity. For each $n\in\N$, let $\gamma_n$ be a geodesic segment with
respect to $\dist$ from $y$ to $z_n$. Choose $r>\dist(b,y)+\xi(y)$.
It follows from assumption~\ref{ass:topology} that the
function $t\mapsto\symdist(y,\gamma_n(t))$ is continuous for each $n\in\N$.
Note that this function is defined on a closed interval and takes the
value $0$ at one endpoint and $\symdist(y,z_n)$ at the other.
Therefore, for $n$ large enough, we may find $t_n\in\Rplus$ such that
$\symdist(y,x_n)=r$, where $x_n:=\gamma_n(t_n)$.

Since $\symdist$ is proper, and all the $x_n$ lie in a closed ball of radius
$r$, we may assume, by taking a subsequence if necessary, that $x_n$ converges
to some point $x\in X$.

Observe that
$\distfn_{z_n}(x_n) = \distfn_{z_n}(y) - \dist(y, x_n)$,
for all $n\in\N$.
Since the $\distfn_{z_n}$ are $1$-Lipschitz with respect to $\symdist$, we may
take limits and get $\xi(x) = \xi(y) - \dist(y,x)$.
On the other hand, $\distfn_{y}(x) = \dist(x,y)-\dist(b,y)$.
So $\distfn_{y}(x) - \xi(x) = r - \dist(b,y) - \xi(y) > 0$.
This shows that $\xi$ is distinct from $\distfn_{y}$.

Now let $p_n$ be a sequence in $X$ such that $\distfn_{p_n}$ converges
to $\distfn_{p}$ in $\distfn(X)$. From what we have just shown,
$p_n$ cannot have any subsequence escaping to infinity.
Therefore $p_n$ is bounded in the $\symdist$ metric.
It then follows from the compactness of closed balls and the continuity and
injectivity of $\distfn$ that $p_n$ converges to $p$.
\end{proof}

We henceforth identify $X$ with its image.

\begin{proposition}
\label{prop:to_infinity}
Assume~\ref{ass:proper}, \ref{ass:geodesic}, and~\ref{ass:topology} hold.
Let $x_n$ be a sequence in $X$ converging to a horofunction.
Then, only finitely many points of $x_n$ lie in any closed ball of
$\symdist$.
\end{proposition}
\begin{proof}
Suppose $x_n$ is a sequence in $X$ such that some subsequence of
$\symdist(b,x_n)$ is bounded.
By taking a further subsequence if necessary, we may assume that
$x_n$ converges to a point $x$ in $X$. By Proposition~\ref{prop:embedding},
$\distfn_{x_n}$ converges to $\distfn_{x}$, and so $x_n$ does not converge
to a horofunction.
\end{proof}

A path $\gamma\colon\Rplus\to X$
is called an \emph{almost-geodesic} if, for each $\epsilon>0$,\index{almost-geodesic}
\begin{equation*}
|\dist(\gamma(0),\gamma(s))+\dist(\gamma(s),\gamma(t))-t|<\epsilon,
\text{\quad for $s$ and $t$ large enough, with $s\le t$}.
\end{equation*}
Rieffel~\cite{rieffel_group} proved that
every almost-geodesic converges to a limit in $X(\infty)$.
A horofunction is called a \emph{Busemann point}\index{Busemann point}
if there exists an almost-geodesic converging to it.
We denote by $X_B(\infty)$ the set of all Busemann points in $X(\infty)$.

Isometries between possibly non-symmetric metric spaces extend continuously
to homeomorphisms between the horofunction compactifications.
\begin{proposition}
\label{prop:transform_horo}
Assume that $\iso$ is an isometry from one possibly non-symmetric metric
space $(X,d)$ to another $(X',d')$, with base-points $b$ and $b'$,
respectively. Then, for every horofunction $\xi$ and point $x\in X$,
\begin{align*}
\iso\cdot\xi(x) = \xi(\iso^{-1}(x)) - \xi(\iso^{-1}(b')),
\end{align*}
\end{proposition}
\begin{proof}
Let $x_n$ be a sequence in $X$ converging to $\xi$.
We have
\begin{align*}
\iso\cdot\xi(x)
   &= \lim_{n\to\infty} d'(x,\iso(x_n)) - d'(b',\iso(x_n)) \\
   &= \lim_{n\to\infty} \Big(d(\iso^{-1}(x),x_n) - d(b,x_n)\Big)
                    + \Big(d(b,x_n) - d(\iso^{-1}(b'),x_n)\Big) \\
   &= \xi(\iso^{-1}(x)) - \xi(\iso^{-1}(b')).
\qedhere
\end{align*}
\end{proof}

\section{The horoboundary of Thurston's Lipschitz metric}
\label{sec:thurston_horo}

We start with a general lemma relating joint continuity\index{joint continuity}
to uniform convergence on compact sets.

\begin{lemma}
\label{lem:uniform_inter}
Let $X$ and $Y$ be two topological spaces
and let $i:X\times Y\to \R$ be a continuous function.
Let $x_n$ be a sequence in $X$ converging to $x\in X$.
Then, $i(x_n,\cdot)$ converges to
$i(x,\cdot)$ uniformly on compact sets of $Y$.
\end{lemma}

\begin{proof}
Take any $\epsilon>0$, and let $K$ be a compact subset of $Y$.
The function $i(\cdot,\cdot)$ is continuous, and so,
for any $y\in Y$, there exists an open neighbourhood
$U_y\subset X$ of $x$ and an open neighbourhood $V_y\subset Y$ of $y$ such that
$|i(x',y') - i(x,y)| < \epsilon$ for all $x'\in U_y$ and $y'\in V_y$.
Since $K$ is covered by $\{V_y \mid y\in K\}$, there exists a finite
sub-covering $\{V_{y_1},\dots,V_{y_n}\}$. Define $U:=\bigcap_{i} U_{y_i}$.
This is an open neighbourhood of $x$, and $|i(x',y) - i(x,y)| < \epsilon$
for all $y\in K$ and $x'\in U$.
\end{proof}

We use Bonahon's theory of geodesic currents~\cite{bonahon_currents}.\index{geodesic current}
The space of geodesic currents is a completion of the space of
homotopy classes of curves on $S$, not necessarily simple, equipped with
positive weights.

More formally, let $G$ be the space of geodesics on the universal cover\index{universal cover}
of $S$, endowed with the compact-open topology.
A geodesic current is a positive measure on $G$ that is invariant under the
action of the fundamental group of $S$.

It is convenient to work with the space of geodesic currents
because both Teichm\"uller space $\teichmullerspace$ and the space
$\measuredlams$ of compactly supported measured geodesic
laminations are embedded into it in a very natural way.\index{measured laminations}
Furthermore, there is a continuous symmetric bilinear form $i(x,y)$
on this space that restricts to the usual intersection form when
$x$ and $y$ are in $\measuredlams$, and takes the value $i(x,y)=\hlength_x(y)$
when $x\in\teichmullerspace$ and $y\in\measuredlams$.

We denote by $\projmeasuredlams$ the projective space of $\measuredlams$,\index{projective measured laminations}
that is, the quotient of $\measuredlams$ by the multiplicative action
of the positive real numbers.
We use $[\mu]$ to denote the equivalence class of $\mu\in\measuredlams$
in $\projmeasuredlams$.
We may identify $\projmeasuredlams$ with the cross-section
$\unitlams:=\{\mu\in\measuredlams \mid \hlength_b(\mu)=1\}$.
We have the following two formulas for the Lipschitz metric.
\begin{align}
\label{eqn:dist_formula}
\stretchdist(x,y)
   = \log \sup_{\eta\in\measuredlams}
        \frac{\hlength_y(\eta)}{\hlength_x(\eta)}
   = \log \sup_{\eta\in\unitlams}
        \frac{\hlength_y(\eta)}{\hlength_x(\eta)},
\end{align}
The second is very useful because the supremum is taken over a compact set,
and is therefore attained.

Recall that we have chosen a base-point $b$ in $\teichmullerspace$.
Define, for any geodesic current $x$,
\begin{align*}
\lfactor(x) &
   := \sup_{\eta\in\measuredlams}\frac{\bonint(x,\eta)}{\hlength_b(\eta)}
    = \sup_{\eta\in\unitlams}\frac{\bonint(x,\eta)}{\hlength_b(\eta)}
\end{align*}
and
\begin{align*}
\lengthfunc_x & :\measuredlams \to \Rplus:
                      \mu \mapsto \frac{\bonint(x,\mu)}{\lfactor(x)}.
\end{align*}

Let $\thurstoncompact:=\teichmullerspace\union\projmeasuredlams$
be the Thurston compactification of Teichm\"uller space.\index{Thurston compactification}

Identify $\projmeasuredlams$ with $\unitlams$, and consider a sequence
$x_n$ in $\thurstoncompact$. Then, $x_n$ converges to a point $x$
in the Thurston
compactification if and only if there is a sequence $\lambda_n$ of positive
real numbers such that $\lambda_n x_n$ converges to $x$ as a geodesic current.
One can take $\lambda_n$ to be identically $1$ if $x\in\teichmullerspace$.

\begin{proposition}
\label{prop:uniform_lengths}
A sequence $x_n$ in $\thurstoncompact$ converges to a point
$x\in\thurstoncompact$ if and only if $\lengthfunc_{x_n}$ converges to
$\lengthfunc_x$ uniformly on compact sets of $\measuredlams$.
\end{proposition}

\begin{proof}
Assume that $x_n$ converges in the Thurston compactification
to a point $x\in\thurstoncompact$.
This implies that, for some sequence $\lambda_n$ of positive real
numbers, $\lambda_n x_n$ converges to $x$ as a geodesic current.
We now apply Lemma~\ref{lem:uniform_inter} to Bonahon's intersection function
to get that $\bonint(\lambda_n x_n,\cdot)$ converges uniformly on compact sets
of $\measuredlams$ to $\bonint(x,\cdot)$.
Therefore, since $\unitlams$ is compact,
$\lfactor(\lambda_n x_n)$ converges to $\lfactor(x)$,
which is a positive real number. So,
$\lengthfunc_{x_n}(\cdot)=\bonint(\lambda_n x_n,\cdot)/\lfactor(\lambda_n x_n)$
converges to $\lengthfunc_x(\cdot)$ uniformly on compact sets of
$\measuredlams$.

Now assume that $\lengthfunc_{x_n}$ converges to
$\lengthfunc_x$ uniformly on compact sets of $\measuredlams$.
Let $y_n$ be a subsequence of $x_n$ converging in $\thurstoncompact$
to a point $y$. As before, we have that
$\lengthfunc_{y_n}$ converges to $\lengthfunc_y$
uniformly on compact sets of $\measuredlams$.
Combining this with our assumption, we get that $\lengthfunc_y$
and $\lengthfunc_x$ agree.
Therefore, $\bonint(y,\cdot)=\lambda\bonint(x,\cdot)$ for some $\lambda>0$.
It follows that $x$ and $y$ are the same point in
$\thurstoncompact$.
We have shown that every convergent subsequence of $x_n$
converges to $x$, which implies that $x_n$ converges to $x$.
\end{proof}

For each $z\in\thurstoncompact$, define the map
\begin{align*}
\Psi_z(x) := \log \sup_{\eta\in\measuredlams}
                 \frac{\lengthfunc_z(\eta)}{\hlength_x(\eta)},
\qquad\text{for all $x$ in $\teichmullerspace$.}
\end{align*}
Note that, if $z\in\teichmullerspace$, then
$\Psi_z(x)=\stretchdist(x,z)-\stretchdist(b,z)$ for all $x\in\teichmullerspace$.

For $x\in\teichmullerspace$ and $y\in\teichmullerspace\union\unitlams$,
let $\maxset_{xy}$ be the set of elements $\eta$ of $\projmeasuredlams\,
({} =\unitlams)$ where $\lengthfunc_y(\eta) / \lengthfunc_x(\eta)$ is maximal.
\begin{lemma}
\label{lem:optimal_dir}
Let $x_n$ be a sequence of points in $\teichmullerspace$ converging to a
point $[\mu]$ in the Thurston boundary.
Let $y$ be a point in $\thurstoncompact$ satisfying $i(y,\mu)\neq 0$,
and let $\nu_n$ be a sequence
in $\unitlams$ such that $\nu_n\in\maxset_{{x_n}y}$ for all $n\in\N$.
Then, any limit point $\nu\in\unitlams$ of $\nu_n$ satisfies $i(\mu,\nu)=0$.
\end{lemma}

\begin{proof}
Consider the sequence of functions
$F_n(\eta):=\lengthfunc_y(\eta)/\lengthfunc_{x_n}(\eta)$.
By Proposition~\ref{prop:uniform_lengths}, $\lengthfunc_{x_n}$ converges
to $\lengthfunc_\mu$ uniformly on compact sets.
Therefore, for any sequence $\eta_n$ in $\unitlams$ converging to a
limit $\eta$, we have that $F_n(\eta_n)$ converges to
$\lengthfunc_y(\eta)/\lengthfunc_{\mu}(\eta)$
provided $\lengthfunc_y(\eta)$ and $\lengthfunc_{\mu}(\eta)$ are not both zero.
So, by evaluating on a sequence $\eta_n$ converging to $\mu$,
we see that $\sup_{\unitlams} F_n$ converges to $+\infty$.
On the other hand, for any sequence $\eta_n$ converging to some $\eta$
satisfying $i(\mu,\eta)>0$, we get that $F_n(\eta_n)$ converges to something
finite, and so $\eta_n\not\in\maxset_{{x_n}y}$ for $n$ large enough.
The conclusion follows.
\end{proof}

A measured lamination is \emph{maximal} if its support is not
properly\index{maximal measured lamination}
contained in the support of any other measured lamination.
It is \emph{uniquely-ergodic} if every measured lamination with the
same\index{uniquely-ergodic measured lamination}
support is in the same projective class.
Recall that, if $\mu$ is maximal and uniquely-ergodic,
and $\eta\in\measuredlams$ satisfies $i(\mu,\eta)=0$,
then $\mu$ and $\eta$ are proportional~\cite[Lemma~2.1]{diaz_series_lines}.

\begin{lemma}
\label{lem:injective}
The map $\Psi:\thurstoncompact\to C(\teichmullerspace):z\mapsto\Psi_z$
is injective.
\end{lemma}

\begin{proof}
Let $x$ and $y$ be distinct elements of $\thurstoncompact$.
By Proposition~\ref{prop:uniform_lengths}, $\lengthfunc_x$ and
$\lengthfunc_y$ are distinct.
So, by exchanging $x$ and $y$ if necessary, we have
$\lengthfunc_x(\mu)<\lengthfunc_y(\mu)$ for some $\mu\in\unitlams$.
Since $\lengthfunc_x$ and $\lengthfunc_y$ are continuous, we may choose
an open neighbourhood $N$ of $\mu$ in $\unitlams$ small enough that there are
real numbers $u$ and $v$ such that
\begin{align*}
\lengthfunc_x(\eta) \le u < v \le \lengthfunc_y(\eta),
\qquad\text{for all $\eta\in N$}.
\end{align*}

Since the set of maximal uniquely-ergodic measured laminations is dense in
$\unitlams$, we can find such a measured lamination $\mu'$ in $N$
that is not proportional to $x$.
Let $p_n$ be a sequence of points in $\teichmullerspace$ converging to
$[\mu']$ in the Thurston boundary, and let $\nu_n$ be a sequence in
$\unitlams$ such that $\nu_n\in\maxset_{{p_n}x}$ for all $n\in\N$.

Since $\mu'$ is maximal and uniquely ergodic,
$i(\mu',\eta)\neq 0$ for all
$\eta\in\measuredlams$ not proportional to $\mu'$.
So $i(\mu',x)\neq 0$, whether $x$ is in $\teichmullerspace$ or in the
Thurston boundary.

By Lemma~\ref{lem:optimal_dir}, any limit point $\nu\in\unitlams$
of $\nu_n$ satisfies $i(\mu',\nu)=0$, and hence equals $\mu'$.
So, $\nu_n\in N$ for large $n$.
Therefore, by taking $n$ large enough, we can find a point $p$ in
$\teichmullerspace$ such that the supremum of
$\lengthfunc_x(\cdot)/\hlength_p(\cdot)$ is attained in the set $N$.

Putting all this together, we have
\begin{align*}
\sup_{\unitlams}\frac{\lengthfunc_x(\cdot)}{\hlength_p(\cdot)}
  & = \sup_{N}\frac{\lengthfunc_x(\cdot)}{\hlength_p(\cdot)}
   \le \sup_{N}\frac{u}{\hlength_p(\cdot)} \\
  & < \sup_{N}\frac{v}{\hlength_p(\cdot)}
   \le \sup_{N}\frac{\lengthfunc_y(\cdot)}{\hlength_p(\cdot)}
   \le \sup_{\unitlams}\frac{\lengthfunc_y(\cdot)}{\hlength_p(\cdot)}.
\end{align*}
Thus, $\Psi_x(p)<\Psi_y(p)$, which implies that $\Psi_x$ and $\Psi_y$ differ.
\end{proof}

\begin{lemma}
\label{lem:continuous}
The map $\Psi:\thurstoncompact\to C(\teichmullerspace):z\mapsto\Psi_z$
is continuous.
\end{lemma}
\begin{proof}
Let $x_n$ be a sequence in $\thurstoncompact$ converging to a point $x$
also in $\thurstoncompact$.
By Proposition~\ref{prop:uniform_lengths}, $\lengthfunc_{x_n}$ converges
uniformly on compact sets to $\lengthfunc_{x}$.
For all $y\in\teichmullerspace$, the function $\hlength_y$ is bounded
away from zero on $\unitlams$.
We conclude that $\lengthfunc_{x_n}(\cdot)/ \hlength_y(\cdot)$ converges
uniformly on $\unitlams$ to $\lengthfunc_{x}(\cdot)/ \hlength_y(\cdot)$,
for all $y\in\teichmullerspace$. It follows that $\Psi_{x_n}$
converges pointwise to $\Psi_{x}$.
As noted before, this implies that $\Psi_{x_n}$ converges to $\Psi_{x}$
uniformly on bounded sets of $\teichmullerspace$.
\end{proof}

\begin{theorem}
\label{thm:horothurston}
\label{thm:homeo}
The map $\Psi$ is a homeomorphism between the Thurston compactification
and the horofunction compactification of $\teichmullerspace$.
\end{theorem}
\begin{proof}
The injectivity of $\Psi$ was proved in Lemma~\ref{lem:injective}, and so
$\Psi$ is a bijection from $\thurstoncompact$ to its image.
The map $\Psi$ is continuous by Lemma~\ref{lem:continuous}.
As a continuous bijection from a compact space to a Hausdorff one,
$\Psi$ must be a homeomorphism from $\thurstoncompact$
to its image.
So $\Psi(\thurstoncompact)$ is compact and therefore closed.
Using the continuity again, we get
$\Psi(\teichmullerspace)\subset \Psi(\thurstoncompact)
\subset \closure\Psi(\teichmullerspace)$.
Taking closures, we get
$\Psi(\thurstoncompact)=\closure\Psi(\teichmullerspace)$,
which is the horocompactification.
\end{proof}

\section{Horocylic foliations and stretch lines}
\label{sec:busemann_points}

Our goal in this section is to show that every horofunction of
the Lipschitz metric is Busemann. This will be achieved by showing that every
horofunction is the limit of a particular type of geodesic introduced
by Thurston~\cite{thurston_minimal},
called a \emph{stretch line}.\index{stretch line}

Let $\mu$ be a \emph{complete} geodesic lamination.
In other words,\index{complete geodesic lamination}
$\mu$ is not strictly contained within another geodesic lamination,
or equivalently, the complementary regions of $\mu$
are all isometric to open ideal triangles in hyperbolic space.
Note that if the surface $S$ has punctures, then $\mu$ has leaves going out to
the cusps.

\begin{figure}
\input{horofoliation6.pstex_t}
\caption{An ideal triangle foliated by horocycles.}
\label{fig:horofoliation}
\end{figure}
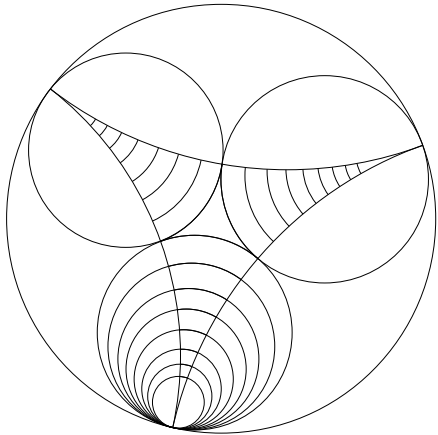

We foliate each of the complementary triangles of $\mu$ with horocyclic arcs
as shown in Figure~\ref{fig:horofoliation}.
The horocyclic arcs meet the boundary of each triangle perpendicularly.
Note that there is a non-foliated region at the center of each triangle, which
is bounded by three horocyclic arcs meeting tangentially. So, the foliation
obtained is actually a partial foliation.
This partial foliation on $S\backslash \mu$ may be extended continuously
to a partial foliation on the whole of $S$.
Given a hyperbolic structure $g$ on $S$, we define a transverse measure
on the partial foliation on $S$ by requiring the measure of every
sub-arc of $\mu$ to be its length in the metric $g$.
The partial foliation with this transverse measure is called the
\emph{horocyclic foliation}, and is denoted $F_\mu(g)$.\index{horocyclic foliation}
Collapsing all non-foliated regions, we obtain
a well-defined element $F_\mu(g)$ of $\measuredfoliations$,
the space of measured foliations on $S$ up to Whitehead equivalence.\index{Whitehead equivalence}
Recall that two measured foliations are said to be Whitehead equivalent 
if one may be deformed to the other by isotopies, deformations that
collapse to points arcs joining a pair of singularities, and the
inverses of such maps.

Note that the horocyclic foliation has around each puncture an annulus
of infinite width foliated by closed leaves parallel to the puncture.
Such a foliation is said to be \emph{trivial around punctures}.
A measured foliation is said to be \emph{totally transverse} to a geodesic\index{totally transverse}
lamination if it is transverse to the lamination and trivial around punctures.
A measured foliation class is said to be totally transverse to a
geodesic lamination if it has a representative that is totally transverse.
Let $\measuredfoliations(\mu)$ be the set of measure classes of
measured foliations that are totally transverse to~$\mu$.
The horocyclic foliation is clearly in $\measuredfoliations(\mu)$.
Thurston proved that the map
$\horocyclic_\mu: \teichmullerspace\to\measuredfoliations(\mu):
   g \mapsto F_\mu(g)$
is in fact a homeomorphism.

The horocyclic foliation gives us a way of deforming the hyperbolic structure
by stretching along $\mu$.
Define the \emph{stretch line} directed by $\mu$ and passing through
$x\in\teichmullerspace$ to be
\begin{align*}
\stretchline_{\mu,x}(t):=\horocyclic^{-1}_\mu(e^t F_\mu(x)),
\qquad\text{for all $t\in\R$}.
\end{align*}
Stretch lines are geodesics for Thurston's Lipschitz metric, that is,
\begin{align*}
\stretchdist(\stretchline_{\mu,x}(s), \stretchline_{\mu,x}(t)) = t-s,
\end{align*}
for all $s$ and $t$ in $\R$ with $s<t$, provided that $\mu$ does not just
consist of geodesics converging at both ends to punctures.

The \emph{stump} of a geodesic lamination is its largest sub-lamination\index{stump}
on which there can be put a compactly-supported transverse measure.
Th\'eret~\cite{theret_negative} showed that a measured foliation class
is totally transverse to a complete geodesic lamination $\mu$ if and only if
its associated measured lamination in $\measuredlams$ transversely meets
every component of the stump of $\mu$.
\begin{theorem}
\label{thm:busemann}
Every point of the horofunction boundary of Thurston's Lipschitz metric
is a Busemann point.\index{Busemann point}
\end{theorem}
\begin{proof}
By Theorem~\ref{thm:horothurston}, the horofunction and Thurston boundaries
coincide.
Let $[\nu]\in\projmeasuredlams$ be any point of the Thurston boundary.
Choose a maximal uniquely-ergodic element $\mu$ of $\measuredlams$ so that
$[\mu]$ is different from $[\nu]$. So, $i(\nu,\mu)>0$.
Take a completion $\overline\mu$ of $\mu$.
The stump of $\overline\mu$ contains the support of $\mu$,
and so must equal this support, since $\mu$ is maximal.
Since $\mu$ is uniquely-ergodic, it has only one component,
which meets $\nu$ transversely.
Let $F$ denote the element of $\measuredfoliations$ associated
to $\nu$.
By~\cite[Lemma~1.8]{theret_negative}, $F$ is totally transverse to
$\overline\mu$.
Therefore the map $t\mapsto\horocyclic^{-1}_{\overline\mu}(e^t F)$ is
a stretch line directed by $\overline\mu$.
It was shown in~\cite{papadopoulos_extension} that such a stretch line
converges in the positive direction to $[\nu]$ in the Thurston
compactification. Since a stretch line is a geodesic, we conclude that
the horofunction $\Psi_\nu$ corresponding to $[\nu]$ is a Busemann point.
\end{proof}

\section{The detour cost}
\label{sec:detour}
\index{detour cost}

Let $(X,d)$ be a possibly non-symmetric metric space with base-point $b$.
We define the \emph{detour cost}
for any two horofunctions $\xi$ and $\eta$ in $X(\infty)$ to be
\begin{align*}
H(\xi,\eta)
   &= \sup_{W\ni\xi} \inf_{x\in W\intersection X} \Big( d(b,x)+\eta(x) \Big),
\end{align*}
where the supremum is taken over all neighbourhoods $W$ of $\xi$ in
$X\cup X(\infty)$.
This concept appears in~\cite{AGW-m}.
An equivalent definition is
\begin{align}\label{eq:3.2}
H(\xi,\eta) &= \inf_{\gamma} \liminf_{t\to\infty}
                 \Big( d(b,\gamma(t))+\eta(\gamma(t)) \Big),
\end{align}
where the infimum is taken over all paths $\gamma:\Rplus\to X$ converging
to $\xi$.

\renewcommand{\theenumi}{(\roman{enumi})}
\renewcommand{\labelenumi}{\theenumi}

\begin{lemma}
\label{lem:triangle_detour}
Let $\xi$ and $\eta$ be horofunctions. Then,
\begin{align*}
\eta(x) \le \xi(x) + H(\xi,\eta),
\qquad\text{for all $x$ in $X$}.
\end{align*}
\end{lemma}
\begin{proof}
By~(\ref{eqn:ideal_tri_ineq}),
\begin{align*}
\eta(x) \le \Big( d(x,z) - d(b,z) \Big) + \Big( d(b,z) + \eta(z) \Big),
\qquad\text{for all $x$ and $z$ in $X$}.
\end{align*}
Note that there is always a path $\gamma$ converging to $\xi$
such that
\begin{align*}
\lim_{t\to\infty} \Big( d(b,\gamma(t))+\eta(\gamma(t)) \Big)=H(\xi,\eta).
\end{align*}
Taking the limit as $z$ moves along such a path gives the result.
\end{proof}

The following was proved in~\cite[Lemma 3.3]{walsh_minimum}.
\begin{lemma}
\label{lem:along_geos}
Let $\gamma$ be an almost-geodesic converging to a Busemann point $\xi$,
and let $y\in X$.
Then,
\begin{equation*}
\lim_{t\to\infty} d(y,\gamma(t)) + \xi(\gamma(t)) = \xi(y).
\end{equation*}
Moreover, for any horofunction $\eta$,
\begin{equation*}
 H(\xi,\eta) = \lim_{t\to\infty} d(b,\gamma(t)) + \eta(\gamma(t)).
\end{equation*}
\end{lemma}
\begin{proof}
Let $\epsilon>0$.
Putting $s=t$ in the definition of almost-geodesic, we see that
\begin{equation*}
| d(\gamma(0),\gamma(t)) - t | < \epsilon,
\qquad\text{for $t$ large enough.}
\end{equation*}
Using this and again the fact that $\gamma$ is an almost-geodesic, we get
\begin{equation*}
| d(\gamma(0),\gamma(s)) + d(\gamma(s),\gamma(t)) - d(\gamma(0),\gamma(t))|
< 2\epsilon,
\end{equation*}
for $s$ and $t$ large enough, with $s\le t$.
Letting $t$ tend to infinity gives
\begin{equation*}
|d(\gamma(0),\gamma(s)) + \xi(\gamma(s)) - \xi(\gamma(0))| \le 2\epsilon,
\qquad\text{for $s$ large enough.}
\end{equation*}
But, since $\gamma$ converges to $\xi$,
\begin{equation*}
|d(y,\gamma(s)) - d(\gamma(0),\gamma(s)) - \xi(y) + \xi(\gamma(0))|
   < \epsilon,
\qquad\text{for $s$ large enough.}
\end{equation*}
Combining these, we deduce the first statement of the lemma.

By Lemma~\ref{lem:triangle_detour},
$\eta(x) \le \xi(x) + H(\xi,\eta)$,
for all $x$ in $X$.
Evaluating at $x=\gamma(s)$, adding $d(b,\gamma(s))$ to both sides, and using
the first part of the lemma with $y=b$, we get
\begin{align*}
\limsup_{s\to\infty} d(b,\gamma(s)) + \eta(\gamma(s)) \le H(\xi,\eta).
\end{align*}
On the other hand, from~(\ref{eq:3.2}),
\begin{equation*}
H(\xi,\eta)
   \le \liminf_{s\to\infty} d(b,\gamma(s)) + \eta(\gamma(s)) .
\end{equation*}
This establishes the second statement of the lemma.
\end{proof}

\begin{proposition}
\label{prop:transformH}
Let $\iso$ be an isometry from one possibly non-symmetric metric
space $(X,d)$ to another $(X',d')$, with base-points $b$ and $b'$ respectively.
Then, the detour costs in $X$ and $X'$ are related by
\begin{align*}
H'(\iso\cdot\xi,\iso\cdot\eta)
   = \xi(\iso^{-1}(b')) + H(\xi,\eta) - \eta( \iso^{-1}(b')),
\text{\qquad for all $\xi,\eta\in X(\infty)$}.
\end{align*}
In particular, every isometry preserves finiteness of the detour cost.
\end{proposition}
\begin{proof}
Let $\xi$ and $\eta$ be in $X(\infty)$.
By Proposition~\ref{prop:transform_horo}, the horofunction $\eta$ is mapped
by $\iso$ to
$\iso\cdot\eta(\cdot) = \eta(\iso^{-1}(\cdot)) - \eta(\iso^{-1}(b'))$.
We have
\begin{align*}
H'(\iso\cdot\xi,\iso\cdot\eta)
   &= \inf_\gamma \liminf_{t\to\infty}
      \Big( d'(b',\iso(\gamma(t)))+\eta(\gamma(t))-\eta(\iso^{-1}(b')) \Big) \\
   &= \inf_\gamma \liminf_{t\to\infty}
        \Big( d(\iso^{-1}(b'),\gamma(t))-d(b,\gamma(t)) \Big) \\
   & \qquad\qquad\qquad\qquad + \Big( d(b,\gamma(t)) + \eta(\gamma(t)) \Big)
          - \eta( \iso^{-1}(b')) \\
   & = \xi(\iso^{-1}(b')) + H(\xi,\eta) - \eta( \iso^{-1}(b')),
\end{align*}
where each time the infimum is taken over all paths in $X$ converging to $\xi$.
\end{proof}
Note that we can take $(X,d)$ and $(X',d')$ to be identical and the isometry
to be the identity map, in which case the proposition says how the detour cost
depends on the base-point:
\begin{align*}
H'(\xi,\eta) = \xi(b') + H(\xi,\eta) - \eta(b').
\end{align*}

For the next proposition, we will need the following assumption.
\begin{assumption}
\label{ass:infinite_dist}
For every sequence $x_n$ in $X$, if $\symdist(b,x_n)$ converges to infinity,
then so does $\dist(b,x_n)$.
\end{assumption}
This assumption is satisfied by the Lipschitz
metric~\cite{papadopoulos_theret_topology}.
\begin{proposition}
\label{prop:zero_on_busemann}
Assume that $(X,d)$ satisfies
assumptions~\ref{ass:proper}, \ref{ass:geodesic},
\ref{ass:topology}, and \ref{ass:infinite_dist}.
If $\xi$ is a horofunction,
then $H(\xi,\xi)=0$ if and only if $\xi$ is Busemann.
\end{proposition}
\begin{proof}
Suppose that $H(\xi,\xi)=0$.
Let $b'\in X$, and let $\xi'=\xi-\xi(b')$ be the horofunction
corresponding to $\xi$ when $b'$ is used as the base-point instead of $b$.
From Proposition~\ref{prop:transformH}, the detour cost with base-point $b'$
satisfies $H'(\xi',\xi')=0$.
So, for any $\epsilon>0$ and neighbourhood $W$ of $\xi'$ in the horofunction
compactification, we may find $x\in W\intersection X$ such that
$|d(b',x) + \xi(x) - \xi(b')| < \epsilon$.

Fix $\epsilon>0$, and let $x_0:=b$.
From the above, we may inductively find a sequence $x_j$ in $X$
converging to $\xi$ such that
\begin{align*}
|d(x_j,x_{j+1}) + \xi(x_{j+1}) - \xi(x_{j})| < \frac{\epsilon}{2^{j+1}},
\qquad\text{for all $j\in\N$}.
\end{align*}
For each $j$, take a finite-length geodesic path $\gamma_j$ from $x_j$
to $x_{j+1}$.

Since $x_j$ converges to a horofunction, we have by
Proposition~\ref{prop:to_infinity} that $\symdist(b,x_j)$ converges to infinity.
Therefore, $\dist(b,x_j)$ also converges to infinity.
But
\begin{align*}
\sum_{j=0}^n \dist(x_j,x_{j+1}) \ge \dist(b,x_{n+1}),
\qquad\text{for all $n\in\N$.}
\end{align*}
So, when we concatenate the geodesic paths $\{\gamma_j\}$, we obtain a
path $\gamma:\Rplus\to X$, defined on the whole of $\Rplus$.

Let $s$ and $t$ be in $\Rplus$, and let $n\in\N$ be such that
\begin{align*}
\sum_{j=0}^{n-1} d(x_j,x_{j+1}) < t \le \sum_{j=0}^{n} d(x_j,x_{j+1}).
\end{align*}
Write $\Delta:=\sum_{j=0}^{n} d(x_j,x_{j+1}) - t$.
So,
\begin{align*}
t  &= \sum_{j=0}^{n} d(x_j,x_{j+1}) - \Delta \\
   &\le -\xi(x_{n+1}) + \epsilon - \Delta \\
   &= (0-\xi(\gamma(s))) + (\xi(\gamma(s))-\xi(\gamma(t))) 
     +(\xi(\gamma(t))-\xi(x_{n+1})) + \epsilon - \Delta.
\end{align*}
Using~(\ref{eqn:ideal_tri_ineq}),
we get
\begin{align*}
t  &\le d(b,\gamma(s)) + d(\gamma(s),\gamma(t)) + d(\gamma(t),x_{n+1})
              + \epsilon - \Delta \\
   &= d(b,\gamma(s)) + d(\gamma(s),\gamma(t)) + \epsilon.
\end{align*}
Since $\gamma$ is a concatenation of geodesic segments,
\begin{align*}
d(b,\gamma(s)) + d(\gamma(s),\gamma(t)) \le s + (t-s) = t.
\end{align*}
Therefore, $\gamma$ is an almost-geodesic, and so it converges,
necessarily to $\xi$ since it passes through each of the points $x_j$.
So, $\xi$ is a Busemann point.

Now assume that $\xi$ is a Busemann point. So, there exists an almost-geodesic
converging to $\xi$.
It follows from Lemma~\ref{lem:along_geos} that $H(\xi,\xi)=0$.
\end{proof}

\begin{proposition}
\label{prop:H_properies}
For all horofunctions $\xi$, $\eta$, and $\nu$,
\begin{enumerate}
\item
\label{itema}
$H(\xi,\eta) \ge 0$;
\item
\label{itemb}
$H(\xi,\nu) \le H(\xi,\eta) + H(\eta,\nu)$.
\end{enumerate}
\end{proposition}
\begin{proof}
\ref{itema}
From~(\ref{eqn:ideal_tri_ineq}), we get that $d(b,y)+\eta(y) \ge 0$,
for all $y$ in $X$.
We conclude that $H(\xi,\eta)$ is non-negative.

\ref{itemb}
By Lemma~\ref{lem:triangle_detour},
\begin{align*}
d(b,x) + \nu(x) \le d(b,x) + \eta(x) + H(\eta,\nu),
\qquad\text{for all $x\in X$}.
\end{align*}
It follows from this that $H(\xi,\nu) \le H(\xi,\eta) + H(\eta,\nu)$.
\end{proof}

By symmetrising the detour cost, the set of Busemann points can be equipped
with a metric. For $\xi$ and $\eta$ in $X_B(\infty)$, let
\begin{equation}\label{eq:3.4}
\delta(\xi,\eta) := H(\xi,\eta)+H(\eta,\xi).
\end{equation}
We call $\delta$ the \emph{detour metric}.\index{detour metric}
This construction appears in~\cite[Remark~5.2]{AGW-m}.
\begin{proposition}
\label{prop:delta_metric}
The function $\delta\colon X_B(\infty)\times X_B(\infty)\to [0,\infty]$
is a (possibly $\infty$-valued) metric.
\end{proposition}
\begin{proof}
The only metric space axiom that does not follow from
Propositions~\ref{prop:zero_on_busemann} and~\ref{prop:H_properies},
and the symmetry of the definition of $\delta$
is that if $\delta(\xi,\eta)=0$ for Busemann points $\xi$ and $\eta$,
then these two points are identical. So, assume this equation holds.

By~\ref{itema} of Proposition~\ref{prop:H_properies},
both $H(\xi,\eta)$ and $H(\eta,\xi)$ are zero.
Applying Lemma~\ref{lem:triangle_detour} twice, we get that $\xi(x)=\eta(x)$
for all $x\in X$.
\end{proof}

The following proposition shows that each isometry of $X$ induces an isometry
on $X_B(\infty)$ endowed with the detour metric.
The independence of the base-point was observed in~\cite[Remark~5.2]{AGW-m}.
\begin{proposition}
\label{prop:isometry_detour_metric}
Let $\iso$ be an isometry from one possibly non-symmetric metric
space $(X,d)$ to another $(X',d')$, with base-points $b$ and $b'$ respectively.
Then, the detour metrics in $X$ and $X'$ are related by
\begin{align*}
\delta'(\iso\cdot\xi,\iso\cdot\eta)
   = \delta(\xi,\eta),
\text{\qquad for all $\xi,\eta\in X(\infty)$}.
\end{align*}
In particular, the detour metric does not depend on the base-point.
\end{proposition}
\begin{proof}
The first part follows from Proposition~\ref{prop:transformH}.
The second part then follows by taking $X=X'$ and $d=d'$,
with $f$ the identity map.
\end{proof}

\section{The detour cost for the Lipschitz metric}
\label{sec:thurston_detour}

We will now calculate the detour cost for Thurston's Lipschitz metric.
This result will be crucial for our study of the isometry
group in the next section.

If $\xi_j$ are a finite set of measured laminations pairwise having
zero intersection number, then we define their sum $\sum_j \xi_j$ to be the
measured lamination obtained by taking the union of the supports and
endowing it with the sum of the transverse measures.
A measured lamination is said to be~\emph{ergodic} if it is non-trivial
and cannot be written as a sum of projectively-distinct non-trivial
measured laminations.
Each measured lamination $\xi$ can be written in one way as the sum of a
finite set of projectively-distinct ergodic measured laminations.
We call these laminations the \emph{ergodic components} of $\xi$.\index{ergodic components}

Let $\blam\in\measuredlams$ be expressed as $\blam= \sum_j \blam_j$ in terms
of its ergodic components.
For $\slam\in\measuredlams$, we write $\slam\ll\blam$
if $\slam$ can be expressed as $\slam = \sum_j f_j \blam_j$, where each $f_j$
is in $\Rplus$.

Recall that $\Psi_\blam$ denotes the horofunction associated to the projective
class of the measured lamination $\blam$.
\begin{theorem}
\label{thm:detourcost}
Let $\slam$ and $\blam$ be measured laminations. If $\slam\ll\blam$, then
\begin{align*}
H(\Psi_\blam,\Psi_\slam) =
      \log \sup_{\eta\in\measuredlams} \frac{i(\blam,\eta)}{\hlength_b(\eta)}
     + \log\max_j(f_j)
     - \log\sup_{\eta\in\measuredlams} \frac{i(\slam,\eta)}{\hlength_b(\eta)},
\end{align*}
where $\slam$ is expressed as $\slam = \sum_j f_j \blam_j$
in terms of the ergodic components $\blam_j$ of $\blam$.
If $\slam\not\ll\blam$, then $H(\Psi_\blam,\Psi_\slam) = +\infty$.
\end{theorem}

\begin{remark}
Here, and in similar situations, we interpret the supremum to be over
the set where the ratio is well defined, that is, excluding values of $\eta$
for which both the numerator and the denominator are zero.
\end{remark}

The proof of this theorem will require several lemmas.

Consider a measured foliation $(F,\mu)$.
Each connected component of the complement in $S$ of the union of the compact
leaves of $F$ joining singularities is either an annulus swept out by
closed leaves or a minimal component in which all leaves are dense.
We call the latter components the \emph{minimal domains}
of $F$.\index{minimal domains}

Denote by $\sing F$ the set of singularities of $F$.
A curve $\curve$ is \emph{quasi-transverse}\index{quasi-transverse}
to $F$ if each connected component
of $\curve\backslash\sing F$ is either a leaf or is transverse to $F$,
and, in a neighbourhood of each singularity, no transverse arc lies in
a sector adjacent to an arc contained in a leaf.
Quasi-transverse curves minimise the total variation of the transverse measure
in their homotopy class, in other words, $\mu(\curve) = i((F,\mu),\curve)$,
for every curve $\curve$ quasi-transverse to $(F,\mu)$;
see~\cite{fathi_laudenbach_poenaru}.

\begin{lemma}
\label{lem:alongminimal}
Let $\epsilon>0$, and let $x_1$ and $x_2$ be two points on the boundary
of a minimal domain $D$ of a measured foliation $(F,\mu)$.
Then, there exists a curve segment $\sigma$ going from $x_1$ to $x_2$
that is contained in $D$ and quasi-transverse to $F$, such that
$\mu(\sigma)<\epsilon$. Moreover, $\sigma$ can be chosen to have a non-trivial
initial segment and terminal segment that are transverse to $F$.
\end{lemma}
\begin{proof}
Let $\tau_1$ and $\tau_2$ be two non-intersecting transverse arcs in $D$
starting at $x_1$ and $x_2$ respectively, parameterised so that
$\mu(\tau_1[0,s]) = \mu(\tau_2[0,s]) = s$, for all $s$.
We also require that the lengths of $\tau_1$ and $\tau_2$
with respect to $\mu$ are less than $\epsilon/2$.

Choose a point on $\tau_1$ and a direction, either left or right, such that
the chosen half-leaf $\gamma$ is infinite, that is, does not hit a singularity.
Since $D$ is a minimal component, $\gamma[0,\infty)$ is dense in $D$.
Let $t_2$ be the first time $\gamma$ intersects $\tau_2$, and let $t_1$
be the last time before $t_2$ that $\gamma$ intersects $\tau_1$.

Let $s_1$ and $s_2$ be such that $\tau_1(s_1)=\gamma(t_1)$ and
$\tau_2(s_2)=\gamma(t_2)$. We may assume that $s_1$ is different from $s_2$,
for otherwise, continue along $\gamma$ until the next time $t_3$ it intersects
$\tau_1[0,s_1)\union\tau_2[0,s_2)$.
If the intersection is with $\tau_1[0,s_1)$, then we take the leaf segment
$\gamma[t_2,t_3)$ with the reverse orientation.
If the intersection is with $\tau_2[0,s_2)$, then we take the leaf segment
$\gamma[t_1,t_3)$. In either case, after redefining $s_1$ and $s_2$ so that
our chosen leaf segment starts at $\tau_1(s_1)$ and ends at $\tau_2(s_2)$,
we get that $s_1\neq s_2$.

If $\tau_1[0,s_1]$ and $\tau_2[0,s_2]$ leave $\gamma$ on opposite sides, then,
by perturbing the concatenation $\tau_1[0,s_1]*\gamma[t_1,t_2]*\tau_2[s_2,0]$,
we can find a curve segment $\sigma$ in $D$
that passes through $x_1$ and $x_2$ and is transverse to $F$ with weight
$\mu(\sigma) = s_1 + s_2 < \epsilon$; see Figure~\ref{fig:figA}.

\begin{figure}
\input{figureA.pstex_t}
\caption{}
\label{fig:figA}
\end{figure}

\begin{figure}
\input{figureB.pstex_t}
\caption{}
\label{fig:figB}
\end{figure}

If $\tau_1[0,s_1]$ and $\tau_2[0,s_2]$ leave $\gamma$ on the same side,
then we apply~\cite[Theorem 5.4]{fathi_laudenbach_poenaru} to get an arc
$\gamma'$ parallel to $\gamma$, contained in a union of a finite number of
leaves and singularities, with endpoints $\gamma'(t_1)=\tau_1(s_1')$
and $\gamma'(t_2)=\tau_2(s_2')$ contained in $\tau_1[0,s_1)$
and $\tau_2[0,s_2)$, respectively (see Figure~\ref{fig:figB}).
Since $s_1\neq s_2$, the points $x_1$, $x_2$, $\tau_2(s_2)$, and $\tau_1(s_1)$
do not form a rectangle foliated by leaves.
Hence the endpoints of $\gamma'$ are not $x_1$ and $x_2$.
As before, by perturbing $\tau_1[0,s_1']*\gamma'[t_1,t_2]*\tau_2[s_2',0]$,
it is easy to construct a curve with the
required properties.
\end{proof}

\begin{lemma}
\label{lem:cross_one}
Let $\xi_j; j\in\{0,\dots,J\}$ be a finite set of ergodic measured laminations
that pairwise have zero intersection number, such that no two are in the same
projective class, and let $C>0$.
Then, there exists a curve $\alpha\in\curves$ such that
$i(\xi_0,\alpha) > C i(\xi_j,\alpha)$ for all $j\in \{1,\dots,J\}$.
\end{lemma}
\begin{proof}

\begin{figure}
\input{crossings.pstex_t}
\caption{Diagram for the proof of Lemma~\ref{lem:cross_one}.}
\label{fig:crossings}
\end{figure}
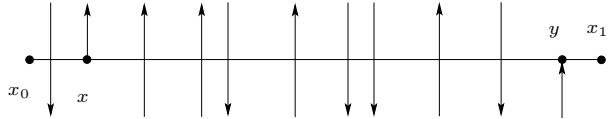

\begin{figure}
\input{rectangle.pstex_t}
\caption{Diagram for the proof of Lemma~\ref{lem:cross_one}.}
\label{fig:rectangle}
\end{figure}
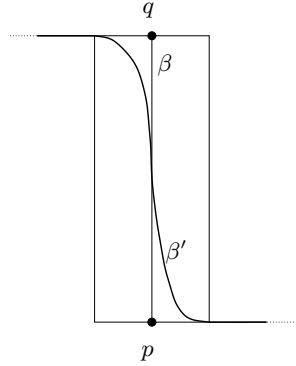

Since the $\xi_j$ do not intersect, we may combine them to form a
measured lamination $\xi:=\sum_j \xi_j$.

Consider first the case where $\xi_0$ is not a curve.
Take a representative $(F,\mu)$ of the element of $\measuredfoliations$
corresponding to $\xi$ by the well known bijection between $\measuredlams$
and $\measuredfoliations$. The decomposition of $\xi$ into a sum of
ergodic measured laminations corresponds to the decomposition of $\mu$
into a sum of partial measured foliations $(F,\mu_j)$.
Each $\mu_j$ is supported on either an annulus of closed leaves of $F$
(if $\xi_j$ is a curve), or a minimal domain.
For each $j$, let $F_j:=(F,\mu_j)$.

Let $I'$ be a transverse arc contained in the interior of the minimal
domain on which $\mu_0$ is supported.
Write $\mu^c := \sum_{j=1}^{J} \mu_j = \mu - \mu_0$ and $F^c:=(F,\mu^c)$.
The measures $\{\mu_j\}$ are mutually singular, and so there is a Borel
subset $X$ of $I'$ such that $\mu_0(X)=\mu_0(I')$ and $\mu^c(X)=0$.
But $X$ may be approximated from above by open sets of $I'$:
\begin{align*}
\mu^c(X)
   = \inf\{ \mu^c(U) \mid \text{$X\subset U\subset I'$ and $U$ is open} \}.
\end{align*}
Therefore, we can find an open set $U$ such that $X\subset U\subset I'$ and
$\mu^c(U) < \mu_0(I') / C$. Since $U$ is open, it is the disjoint union of a
countable collection of open intervals.
Since $\mu_0(U)=\mu_0(I')>\mu^c(U) C$, at least one of these intervals $I$
must satisfy $\mu_0(I) > \mu^c(I) C$.
Choose $\epsilon >0$ such that $\mu_0(I)-\epsilon > \mu^c(I) C$.

For $x$ and $y$ in $I$, we denote by $[x,y]$ the closed sub-arc of $I$
connecting $x$ and $y$. Open and half-open sub-arcs are denoted in an
analogous way. Let $x_0$ and $x_1$ be the endpoints of $I$.

Choose a point $x\in I$ such that $\leb[x_0,x] < \epsilon/3$ and there is
an infinite half-leaf $\gamma$ of $F$ starting from $\gamma(0)=x$.
So, we may go along $\gamma$ until we reach a point $y:=\gamma(t)$ in $I$
such that $\leb[y,x_1] < \epsilon/3$.

Let $B$ be the set of intervals $[p,q)$ in $I$ such that the finite leaf
segment $\gamma[0,t]$ crosses $I$ at $p$ and at $q$, the two crossings
are in the same direction, and $\gamma$ does not cross $I$ in the interval
$(p,q)$.

Consider an element $[p,q)$ of $B$. Assume that $\gamma$ passes through
$p$ before it passes through $q$, that is, $\gamma(t_p)=p$ and $\gamma(t_q)=q$
for some $t_p$ and $t_q$ in $[0,t]$, with $t_p<t_q$.
(The other case is handled similarly.)
Let $\scurv$ be the closed curve consisting of $\gamma[t_p,t_q]$ concatenated
with the sub-interval $[p,q]$ of $I$. Since $\gamma$ contains no singular
point of $F$, there exists a narrow rectangular neighbourhood of
$\gamma[t_p,t_q]$ not containing
any singular point of $F$, and so we may perturb $\scurv$ to get a closed curve
$\scurv'$ that is transverse to $F$ (see Figure~\ref{fig:rectangle}).
We have
\begin{align*}
i(\scurv,F_j) = i(\scurv',F_j) = \mu_j(\scurv') = \mu_j[p,q),
\end{align*}
for all $j\in\{0,\dots,J\}$.
The second equality uses the fact that $\scurv'$ is transverse to $F_j$,
and hence quasi-transverse.
Let $Z$ be the set of curves $\scurv$ obtained in this way from the elements
$[p,q)$ of $B$.

The set $[x,y) \backslash\cup B$ is composed of a finite number of intervals
of the form $[r,s)$, where $\gamma[0,t]$ crosses $I$ at $r$ and $s$ in different
directions and does not cross the interval $(r,s)$.
From $\gamma(t)$, we continue along $\gamma$ until the first time $t'$
that $\gamma$ crosses one of the intervals $[r,s)$ comprising
$[x,y) \backslash\cup B$.
The direction of this crossing will be the same as either that at $r$ or that
at $s$. We assume the former case; the other case is similar.
As before, we have that the curve $\beta$ formed from the segment of $\gamma$
going from $r$ to $\gamma(t')$, concatenated with the sub-arc $[r,\gamma(t')]$
of $I$ satisfies $i(\scurv,F_j) = \mu_j[r,\gamma(t'))$, for all $j$.
We add $[r,\gamma(t'))$ to the set $B$, and $\scurv$ to the set $Z$.
Observe that $[x,y)\backslash \cup B$ remains composed of the same number
of intervals of the same form, only now one of them is shorter.

We continue in this manner, adding intervals to $B$ and curves to $Z$.
Since $\gamma$ crosses $I$ on a dense subset of $I$, the maximum $\mu$-measure
of the component intervals of $[x,y)\backslash \cup B$ can be made as
small as we wish. But there are a fixed number of these components, and so
we can make $\mu([x,y)\backslash \cup B)$ as small as we wish.
We make it smaller than $\epsilon/3$.

We have
\begin{align*}
\max_{\scurv\in Z} \frac{i(\scurv,F_0)}{i(\scurv,F^c)}
   &\ge \frac{\sum_{\scurv\in Z}i(\scurv,F_0)}{\sum_{\scurv\in Z}i(\scurv,F^c)} \\
   &= \frac{\mu_0(\cup B)}{\mu^c(\cup B)} \\
   &\ge \frac{\mu_0(I)-\epsilon}{\mu^c(I)} \\
   &> C.
\end{align*}
So, some curve $\scurv\in Z$ satisfies
$i(\scurv,\xi_0) = i(\scurv,F_0) > C i(\scurv,F^c) \ge C i(\scurv, \xi_j)$,
for all $j\in\{1,\dots,J\}$.

Now consider the case where $\xi_0$ is a curve.
A slight adaption of the proof
of~\cite[Proposition 3.17]{fathi_laudenbach_poenaru} shows that there is a 
curve $\alpha\in\curves$ having positive intersection number with
$\xi_0$ and zero intersection number with every other curve in the support
of $\xi$.

By~\cite[Proposition 5.9]{fathi_laudenbach_poenaru}, there exists a
measured foliation $(F,\mu)$ representing the element of $\measuredfoliations$
associated to $\xi$ such that $\alpha$ is transverse to $F$ and avoids
its singularities.

Again we use the decomposition of $\mu$ into a sum of mutually-singular
partial measured foliations $(F,\mu_j)$, corresponding to the $\xi_j$.
Since $\alpha$ is transverse to $F$, we have, for each $j$, that
$i(\alpha,\xi_j) = \mu_j(\alpha)$, where $\mu_j(\alpha)$ denotes the total
mass of $\alpha$ with respect to the transverse measure $\mu_j$.
It follows that $\alpha$ crosses the annulus $A$ associated to $\xi_0$
at least once, but never enters any of the annuli associated to the other
curves in the support of~$\xi$.

Consider the following directed graph.
We take a vertex for every minimal domain of $F$
through which $\alpha$ passes, and for every time $\alpha$ crosses $A$.
So, there is at most one vertex associated to each minimal domain but there
may be more than one associated to~$A$.
As we move along the curve $\alpha$, we get a cyclic sequence of these
vertices. We draw a directed edge between each vertex of this cyclic
sequence and the succeeding one, and label it with the point of $\curves$
where $\alpha$ leaves the minimal domain or annulus and enters the next.
There may be more than one directed edge between a pair of vertices,
but each will have a different label.

The curve $\alpha$ induces a circuit in this directed graph.
Choose a simple sub-circuit $c$ that passes through at least
one vertex associated to a crossing of~$A$.
Here, simple means that no vertex is visited more than once.

Construct a curve in $S$ as follows.
For each vertex passed through by $c$ associated to a crossing of $A$,
take the associated segment of $\alpha$. For each
vertex passed through by $c$ associated to a minimal domain, choose $\epsilon$
to be less than the height of $A$ divided by $C$, and take the curve
segment given by Lemma~\ref{lem:alongminimal} passing through the minimal
domain, joining the points labeling the incoming and outgoing directed edges.
When we concatenate all these curve segments, we get a curve $\alpha'$
that passes through $A$, and that is quasi-transverse to $F$. Furthermore,
$i(\xi_j,\alpha') <\epsilon$ for each non-curve component $\xi_j$ of $\xi$,
and $i(\xi_j,\alpha')=0$ for each curve component different from $\xi_0$.
The conclusion follows.
\end{proof}

\begin{lemma}
\label{lem:max_intersection_ratio}
Let $\slam$ and $\blam$ be measured laminations. If $\slam\ll\blam$, then
\begin{align*}
\sup_{\eta\in\measuredlams} \frac{i(\slam,\eta)}{i(\blam,\eta)}
   = \max_j(f_j),
\end{align*}
where $\slam$ is expressed as $\slam = \sum_j f_j \blam_j$
in terms of the ergodic components $\blam_j$ of $\blam$.
If $\slam\not\ll\blam$, then the supremum is $+\infty$.
\end{lemma}

\begin{proof}
In the case where $i(\slam,\blam)>0$, we take $\eta:=\blam$ to get that
the supremum is~$\infty$.

So assume that $i(\slam,\blam)=0$.
In this case, we can write $\slam=\sum_j f_j \xi_j$ and
$\blam=\sum_j g_j \xi_j$, where the $\xi_j; j\in\{0,\dots,J\}$ are
a finite set of ergodic measured laminations pairwise having zero intersection
number,
and the $f_j$ and $g_j$ are non-negative coefficients such that, for all $j$,
either $f_j$ or $g_j$ is positive.
Relabel the indexes so that $\max_j(f_j/g_j)=f_0/g_0$.
We have
\begin{align*}
I(\eta) := \frac{i(\slam,\eta)}{i(\blam,\eta)}
   = \frac{\sum_j f_j i(\xi_j,\eta)}{\sum_j g_j i(\xi_j,\eta)}.
\end{align*}
Simple algebra establishes that $I(\eta)\le f_0/g_0$ for all $\eta$.

For each $C>0$, we apply Lemma~\ref{lem:cross_one}
to get a curve $\alpha_C$ such that $i(\xi_0,\alpha_C) > C i(\xi_j,\alpha_C)$
for all $j\in\{1,\dots,J\}$.
By choosing $C$ large enough, we can make $I(\alpha_C)$
as close as we like to $f_0/g_0$.

We conclude that $\sup_\eta I(\eta)=f_0/g_0$.
\end{proof}

\begin{proof}[Proof of Theorem~\ref{thm:detourcost}]
Let $F$ be the measured foliation corresponding to $\mu$.
So, $i(F,\alpha)=i(\mu,\alpha)$, for all $\alpha\in\curves$.
As in the proof of Theorem~\ref{thm:busemann}, we may find
a complete geodesic lamination $\mu'$ that is totally transverse to $F$.
Consider the stretch line $\geo(t):= \horocyclic_{\mu'}^{-1}(e^t F)$
directed by $\mu'$ and passing through $\horocyclic_{\mu'}^{-1}(F)$.

From the results in~\cite{papadopoulos_extension}, $\geo$ converges in the
positive direction to $[\mu]$ in the Thurston boundary.
So, by Theorem~\ref{thm:horothurston}, $\geo$ converges to $\Psi_\mu$
in the horofunction boundary. Therefore, since $\geo$ is a geodesic,
\begin{align*}
H(\Psi_\mu,\Psi_\nu)
   &= \lim_{t\to\infty}
             \Big(\stretchdist(b,\geo(t)) + \Psi_\nu(\geo(t))\Big) \\
   &= \lim_{t\to\infty} \Big(\log \sup_{\eta\in\measuredlams}
             \frac{\hlength_{\geo(t)}(\eta)}{\hlength_b(\eta)}
    + \log \sup_{\eta\in\measuredlams}
             \frac{i(\nu,\eta)}{\hlength_{\geo(t)}(\eta)} \Big)
    - \log \sup_{\eta\in\measuredlams} \frac{i(\nu,\eta)}{\hlength_b(\eta)}.
\end{align*}
From~\cite[Cor.~2]{theret_thesis}, for every $\eta\in\measuredlams$,
there exists a constant $C_\eta$ such that
\begin{align*}
i(\horocyclic_{\mu'}(\geo(t)),\eta) \le \hlength_{\geo(t)}(\eta)
    \le i(\horocyclic_{\mu'}(\geo(t)),\eta) + C_\eta,
\qquad\text{for all $t\ge0$}.
\end{align*}
So,
\begin{align}
\label{eqn:bounds}
i(F,\eta) \le e^{-t}\hlength_{\geo(t)}(\eta)
            \le i(F,\eta) + e^{-t} C_\eta,
\end{align}
for all $\eta\in\measuredlams$ and $t\ge0$.

So, $e^{-t}\hlength_{\geo(t)}$ converges pointwise to $i(F,\cdot)=i(\mu,\cdot)$
on $\measuredlams$. Since $\geo(t)$ converges to $[\mu]$,
we get, by Proposition~\ref{prop:uniform_lengths},
that $\lengthfunc_{\geo(t)}$ converges to $\lengthfunc_{\mu}$
uniformly on compact sets. Combining this with the convergence of
$e^{-t}\hlength_{\geo(t)}$, and evaluating at any measured lamination,
we see that $e^{-t} \lfactor(\geo(t))$ converges
to $\lfactor(\mu)$. Using again the convergence of $\lengthfunc_{\geo(t)}$,
we conclude that $e^{-t}\hlength_{\geo(t)}$ converges to $i(\mu,\cdot)$
uniformly on compact sets.
So,
\begin{align*}
\lim_{t\to\infty} \sup_{\eta\in\measuredlams}
              \frac{e^{-t}\hlength_{\geo(t)}(\eta)}{\hlength_b(\eta)}
   = \sup_{\eta\in\measuredlams} \frac{i(\mu,\eta)}{\hlength_b(\eta)}.
\end{align*}
From the left-hand inequality of~(\ref{eqn:bounds}), we get
\begin{align*}
\sup_{\eta\in\measuredlams} \frac{i(\nu,\eta)}{e^{-t}\hlength_{\geo(t)}(\eta)}
   \le \sup_{\eta\in\measuredlams} \frac{i(\nu,\eta)}{i(\mu,\eta)},
\qquad\text{for $t\ge0$}.
\end{align*}
But the limit of a supremum is trivially greater than or equal to the supremum
of the limits. We conclude that
\begin{align*}
\lim_{t\to\infty} \sup_{\eta\in\measuredlams}
         \frac{i(\nu,\eta)}{e^{-t}\hlength_{\geo(t)}(\eta)}
   = \sup_{\eta\in\measuredlams} \frac{i(\nu,\eta)}{i(\mu,\eta)}.
\end{align*}
The result now follows on applying Lemma~\ref{lem:max_intersection_ratio}.
\end{proof}

\begin{corollary}
\label{cor:detourmetric}
If $\slam$ and $\blam$ in $\measuredlams$ can be written in the form
$\slam =\sum_j f_j \eta_j$ and $\blam =\sum_j g_j \eta_j$, where the $\eta_j$
are ergodic elements of $\measuredlams$ that pairwise have zero intersection
number,
and the $f_j$ and $g_j$ are positive coefficients, then the detour metric
between $\Psi_\slam$ and $\Psi_\blam$ is
\begin{align*}
\delta(\Psi_\slam,\Psi_\blam)
    = \log \max_j \frac{f_i}{g_i} + \log \max_j \frac{g_i}{f_i}.
\end{align*}
If $\slam$ and $\blam$ cannot be simultaneously written in this form, then
$\delta(\Psi_\slam,\Psi_\blam)=+\infty$.
\end{corollary}

\section{Isometries}
\label{sec:isometries}

In this section, we prove Theorems~\ref{thm:isometries} and~\ref{thm:distinct}.

Recall that the curve complex $\curvecomp$ is the simplicial complex having\index{curve complex}
vertex set $\curves$, and where a set of vertices form a simplex when they have
pairwise disjoint representatives.
The automorphisms of the curve complex were characterised by\index{curve complex!automorphisms}
Ivanov~\cite{ivanov_automorphisms}, Korkmaz~\cite{korkmaz_automorphisms},
and Luo~\cite{luo_automorphisms}.\index{Ivanov--Korkmaz--Luo Theorem}
\begin{theorem}[Ivanov-Korkmaz-Luo]
\label{thm:ivanov_automorphism}
Assume that $S$ is not a sphere with four or fewer punctures,
nor a torus with two or fewer punctures. Then all automorphisms of
$\curvecomp$ are given by elements of $\modulargroup$.
\end{theorem}

We will also need the following theorem contained in~\cite{ivanov_isometries},
which was stated there for measured foliations rather than measured
laminations.
For each $\mu\in\measuredlams$, we define the set
$\mu^\perp := \{ \nu\in\measuredlams \mid i(\nu,\mu)=0 \}$.
\begin{theorem}[Ivanov]
\label{thm:ivanov_codimension}
Assume that $S$ is not a sphere with four or fewer punctures,
nor a torus with one or fewer punctures. Then the co-dimension of the set
$\mu^\perp$
in $\measuredlams$ is equal to 1 if and only if $\mu$ is a positive
multiple of a simple closed curve.
\end{theorem}

By Theorem~\ref{thm:horothurston}, the horoboundary can be identified with
the Thurston boundary. So, the homeomorphism induced on the horoboundary by an
isometry of $\teichmullerspace$ may be thought of as a map from
$\projmeasuredlams$ to itself.

\begin{lemma}
\label{lem:preserveszero}
Let $\iso$ be an isometry of the Lipschitz metric.
For all $[\mu_1]$ and $[\mu_2]$ in $\projmeasuredlams$,
we have $i(\iso[\mu_1],\iso[\mu_2]) = 0$ if and only if $i([\mu_1],[\mu_2])=0$.
\end{lemma}
\begin{proof}
For any two elements $[\mu]$ and $[\eta]$ of $\projmeasuredlams$, we have,
by Proposition~\ref{prop:transformH}, that
$H(\Psi_{[\eta]},\Psi_{[\mu]})$ is finite if and only if
$H(\Psi_{\iso[\eta]},\Psi_{\iso[\mu]})$ is finite.
Also, from Theorem~\ref{thm:detourcost},
$H(\Psi_{[\eta]},\Psi_{[\mu]})$ is finite if and only if $[\mu]\ll[\eta]$.
It follows that $\iso$ preserves the relation $\ll$ on $\projmeasuredlams$.

Two elements $[\mu_1]$ and $[\mu_2]$ of $\projmeasuredlams$ satisfy
$i([\mu_1],[\mu_2])=0$ if and only if there is some projective measured
lamination $[\eta]\in\projmeasuredlams$ such that $[\mu_1]\ll[\eta]$
and $[\mu_2]\ll[\eta]$.

We conclude that $[\mu_1]$ and $[\mu_2]$ have zero intersection number
if and only if $f[\mu_1]$ and $f[\mu_2]$ have zero intersection number.
\end{proof}

\begin{lemma}
\label{lem:same_on_PML}
Let $\iso$ be an isometry of $(\teichmullerspace,\stretchdist)$.
Then, there is an extended mapping class that agrees with $\iso$ on
$\projmeasuredlams$.
\end{lemma}
\begin{proof}
By Theorem~\ref{thm:horothurston}, the horoboundary can be identified with
the Thurston boundary.
So, the map induced by $\iso$ on the horoboundary is a homeomorphism of
$\projmeasuredlams$ to itself.

We identify $\projmeasuredlams$ with the level set
$\{\mu\in\measuredlams \mid \hlength_b(\mu)=1\}$.
There is then a unique way of extending the map
$\iso:\projmeasuredlams\to\projmeasuredlams$ to a positively homogeneous map
$\iso_*:\measuredlams\to\measuredlams$. Evidently, $\iso_*$ is
also a homeomorphism.

From Lemma~\ref{lem:preserveszero},
we get that $(\iso_*\mu)^\perp = \iso_*(\mu^\perp)$ for all
$\mu\in\measuredlams$.
By Theorem~\ref{thm:ivanov_codimension}, positive multiples of simple closed
curves can be characterised in $\measuredlams$ as those elements $\mu$
such that the co-dimension of the set $\mu^\perp$ equals $1$.
Since $\iso_*$ is a homeomorphism on $\measuredlams$, it preserves
the co-dimension of sets.
We conclude that $\iso_*$ maps elements of $\measuredlams$ of the form
$\lambda\curve$, where $\lambda>0$ and $\curve\in\curves$,
to elements of the same form.

So, $\iso$ induces a bijective map on the vertices of the curve complex.
From Lemma~\ref{lem:preserveszero} it is clear that there is an edge
between two curves $\curve_1$ and $\curve_2$ in $\curvecomp$
if and only if there is an edge between $\iso\curve_1$ and $\iso\curve_2$.
Using the fact that a set of vertices span a simplex
exactly when every pair of vertices in the set have an edge connecting them,
we see that $\iso$ acts as an automorphism on the curve complex.

We now apply Theorem~\ref{thm:ivanov_automorphism} of Ivanov-Korkmaz-Luo
to deduce that there is some extended mapping class $\mapclass\in\modulargroup$
that agrees with $\iso$ on $\curves$, considered as a subset of
$\projmeasuredlams$. But $\curves$ is dense in $\projmeasuredlams$,
and the actions of $\mapclass$ and $\iso$ on $\projmeasuredlams$ are both continuous.
Therefore, $\mapclass$ and $\iso$ agree on all of $\projmeasuredlams$.
\end{proof}

\begin{lemma}
\label{lem:horofunctions_distinguish_points}
Let $x$ and $y$ be points of $\teichmullerspace$.
If $\Psi_\nu(x)=\Psi_\nu(y)$ for all horofunctions $\Psi_\nu$, then $x=y$.
\end{lemma}

\begin{proof}
Assume that $x$ and $y$ are distinct, and let $\mu(x,y)$ be the maximal
chain-recurrent lamination that maximises stretch between $x$ and $y$.
Thurston showed~\cite[Theorem 8.5]{thurston_minimal}
that there exists a geodesic segment from $x$ to $y$ consisting of a
concatenation of a finite number of segments of stretch lines, each of which
stretches along some complete geodesic lamination containing $\mu(x,y)$.
We extend this geodesic segment to a ray $\gamma$ by continuing along the
final stretch line.

The ray $\gamma$ is geodesic since, along it, $\mu(x,y)$ is maximally stretched.

Let $[\nu]\in\projmeasuredlams$ be the limit in the Thurston boundary
in the forward direction along $\gamma$.
Since $y$ lies on a geodesic ray connecting $x$ to $\nu$,
we have $\Psi_\nu(x) = L(x,y) + \Psi_\nu(y)$.
Thus, $\Psi_\nu(x)$ and $\Psi_\nu(y)$ differ.
\end{proof}

\begin{lemma}
\label{lem:hyperbolic_boundary}
Let $[\nu_1]$ and $[\nu_2]$ be points in $\projmeasuredlams$,
and let $x_n$ be a sequence in $\teichmullerspace$ such that
$\Psi_{\nu_1}(x_n)$ and $\Psi_{\nu_2}(x_n)$ both converge to $-\infty$.
Then, $i(\nu_1,\eta)+i(\nu_2,\eta)=0$ for some
$\eta\in\measuredlams\backslash\{0\}$.
\end{lemma}

\begin{proof}
The function $\Psi_{\nu_1}$ is continuous and therefore bounded below on any
compact set. We deduce that $x_n$ leaves every compact set,
and so $L(x_n,b)$ converges to infinity.
From~(\ref{eqn:dist_formula}),
we get
\begin{align}
\label{eqn:some_length_goes_to_zero}
\sup_{\eta\in\unitlams} \frac{1}{\hlength_{x_n}(\eta)}
   \rightarrow \infty,
\end{align}
as $n$ tends to infinity.

From the convergence of $\Psi_{\nu_1}(x_n)$ and $\Psi_{\nu_2}(x_n)$
to $-\infty$, we get
\begin{align*}
\sup_{\eta\in\unitlams} \frac{i(\nu_k,\eta)}{\hlength_{x_n}(\eta)}
   \rightarrow 0,
\qquad\text{for $k\in\{1,2\}$.}
\end{align*}
It follows that
\begin{align*}
\sup_{\eta\in\unitlams} \frac{i(\nu_1,\eta)+i(\nu_2,\eta)}{\hlength_{x_n}(\eta)}
   \rightarrow 0.
\end{align*}
Observe that this would contradict~(\ref{eqn:some_length_goes_to_zero})
if $i(\nu_1,\cdot)+i(\nu_2,\cdot)$ was bounded below on $\unitlams$
by a positive constant.
The conclusion now follows since $i(\nu_1,\cdot)+i(\nu_2,\cdot)$
attains its infimum on $\unitlams$.
\end{proof}

\begin{lemma}
\label{lem:same_on_T}
Let $\iso$ and $\mapclass$ be two isometries of
$(\teichmullerspace,\stretchdist)$.
If the extensions of $\iso$ and $\mapclass$ coincide on the boundary
$\projmeasuredlams$, then $\iso=\mapclass$.
\end{lemma}

\begin{proof}
The map $g:=\mapclass^{-1}\after\iso$ is an isometry of
$(\teichmullerspace,\stretchdist)$.

Suppose that there exists some point $b$ in $\teichmullerspace$
that is not fixed by $g$.
By Lemma~\ref{lem:horofunctions_distinguish_points}, there is a
horofunction $\Psi_\nu$ such that $\Psi_\nu(g(b))\neq\Psi_\nu(b)$.
By considering the inverse of $g$ if necessary, we may assume that
$\Psi_\nu(g(b))<\Psi_\nu(b)$.
Take $b$ to be the base-point, so that $\Psi_\nu(b)=0$.

Recall that the set of maximal uniquely-ergodic measured laminations is dense
in $\projmeasuredlams$. So, we may find two projectively-distinct maximal
uniquely-ergodic measured laminations $\nu_1$ and $\nu_2$ that are sufficiently
close to $\nu$ that $\Psi_{\nu_1}(g(b))<\Psi_{\nu_1}(b)=0$ and
$\Psi_{\nu_2}(g(b))<\Psi_{\nu_2}(b)=0$.

Consider the sequence of iterates $x_n:= g^n(b)$.

The map $g^{-1}$ extends continuously to the horofunction/Thurston boundary,
and fixes every point of this boundary. So, by
Proposition~\ref{prop:transform_horo},
we have $\xi(\cdot) = \xi(g(\cdot)) - \xi(g(b))$ for all horofunctions $\xi$.

Applying this to $\xi:=\Psi_{\nu_1}$, we get
$\xi(x_n) = \xi(x_{n-1}) + \xi(g(b))$, for all $n\ge 1$.
Therefore, $\xi(x_n) = n\xi(g(b))$ for all $n\ge 1$,
and so $\xi(x_n)$ converges to $-\infty$, as $n$ tends to infinity.

Similarly, $\Psi_{\nu_2}(x_n)$ converges to $-\infty$.

We now apply Lemma~\ref{lem:hyperbolic_boundary} to get that
$i(\nu_1,\eta)+i(\nu_2,\eta)=0$ for some $\eta\in\measuredlams\backslash\{0\}$.
But this is a contradiction, since $\nu_1$ and $\nu_2$ are
projectively-distinct maximal uniquely-ergodic measured laminations.

We deduce that every point of Teichm\"uller space is fixed by $g$,
and so $f=h$.
\end{proof}

\begin{theorem}
\label{thm:isometries}
If $S$ is not a sphere with four or fewer punctures,
nor a torus with two or fewer punctures, then every isometry of
$\teichmullerspace$ equipped with Thurston's Lipschitz metric is an element
of the extended mapping class group $\modulargroup$.
\end{theorem}
\begin{proof}
This is a consequence of Lemmas~\ref{lem:same_on_PML} and~\ref{lem:same_on_T}.
\end{proof}

\begin{theorem}
\label{thm:distinct}
Let $S_{g,n}$ and $S_{g',n'}$ be surfaces of negative Euler characteristic.
Assume $\{(g,n), (g',n')\}$ is different from each of the three sets
\begin{align*}
\{(1,1), (0,4)\}, \qquad \{(1,2), (0,5)\}, \qquad{and}\quad \{(2,0), (0,6)\}.
\end{align*}
If $(g,n)$ and $(g',n')$ are distinct, then the Teichm\"uller spaces
$\teichmullerspc(S_{g,n})$ and $\teichmullerspc(S_{g',n'})$
with their respective Lipschitz metrics are not isometric.
\end{theorem}
\begin{proof}
Let $S_{g,n}$ and $S_{g',n'}$ be two surfaces of negative Euler characteristic,
and let $\iso$ be an isometry between the associated Teichm\"uller spaces
$\teichmullerspc(S_{g,n})$ and $\teichmullerspc(S_{g',n'})$, each endowed
with its respective Lipschitz metric.

Recall that the dimension of $\teichmullerspc(S_{g,n})$ is $6g-6+2n$.
Consider the cases not covered by Ivanov's Theorem~\ref{thm:ivanov_codimension},
namely the sphere with three or four punctures, and the torus with one
puncture. A comparison of dimension shows that $\teichmullerspc(S_{0,3})$
is not isometric to any other Teichm\"uller space, and that
$\teichmullerspc(S_{0,4})$ and $\teichmullerspc(S_{1,1})$ may be isometric
to each another but not to any other Teichm\"uller space.

So assume that neither $\teichmullerspc(S_{g,n})$ nor
$\teichmullerspc(S_{g',n'})$ is one of these exceptional surfaces.

By the same reasoning as in the first part of the proof of
Lemma~\ref{lem:same_on_PML}, we conclude that $\iso$ induces an isomorphism
from the curve complex of $S_{g,n}$ to that of $S_{g',n'}$.
However, according to~\cite[Lemma~2.1]{luo_automorphisms},
the only isomorphisms between curve complexes are
$\curvecmp(S_{1,2})\cong\curvecmp(S_{0,5})$ and
$\curvecmp(S_{2,0})\cong\curvecmp(S_{0,6})$.
The conclusion follows.
\end{proof}


\bibliographystyle{abbrv}
\bibliography{stretch}

\end{document}

%% file: horofoliation6.pstex_t
\begin{picture}(0,0)%
\includegraphics{horofoliation6.pstex}%
\end{picture}%
\setlength{\unitlength}{2368sp}%
\begingroup\makeatletter\ifx\SetFigFont\undefined%
\gdef\SetFigFont#1#2#3#4#5{%
  \reset@font\fontsize{#1}{#2pt}%
  \fontfamily{#3}\fontseries{#4}\fontshape{#5}%
  \selectfont}%
\fi\endgroup%
\begin{picture}(4508,4506)(3002,-7627)
\end{picture}%

%% file: figureA.pstex_t
\begin{picture}(0,0)%
\includegraphics{figureA.pstex}%
\end{picture}%
\setlength{\unitlength}{1184sp}%
\begingroup\makeatletter\ifx\SetFigFont\undefined%
\gdef\SetFigFont#1#2#3#4#5{%
  \reset@font\fontsize{#1}{#2pt}%
  \fontfamily{#3}\fontseries{#4}\fontshape{#5}%
  \selectfont}%
\fi\endgroup%
\begin{picture}(4648,7770)(4189,-7771)
\put(4426,-3136){\makebox(0,0)[lb]{\smash{{\SetFigFont{9}{10.8}{\rmdefault}{\mddefault}{\updefault}{\color[rgb]{0,0,0}$\gamma(t_1)$}%
}}}}
\put(7576,-361){\makebox(0,0)[lb]{\smash{{\SetFigFont{9}{10.8}{\rmdefault}{\mddefault}{\updefault}{\color[rgb]{0,0,0}$x_2$}%
}}}}
\put(4501,-7636){\makebox(0,0)[lb]{\smash{{\SetFigFont{9}{10.8}{\rmdefault}{\mddefault}{\updefault}{\color[rgb]{0,0,0}$x_1$}%
}}}}
\put(7501,-4936){\makebox(0,0)[lb]{\smash{{\SetFigFont{9}{10.8}{\rmdefault}{\mddefault}{\updefault}{\color[rgb]{0,0,0}$\gamma(t_2)$}%
}}}}
\end{picture}%

%% file: figureB.pstex_t
\begin{picture}(0,0)%
\includegraphics{figureB.pstex}%
\end{picture}%
\setlength{\unitlength}{1579sp}%
\begingroup\makeatletter\ifx\SetFigFont\undefined%
\gdef\SetFigFont#1#2#3#4#5{%
  \reset@font\fontsize{#1}{#2pt}%
  \fontfamily{#3}\fontseries{#4}\fontshape{#5}%
  \selectfont}%
\fi\endgroup%
\begin{picture}(6780,5286)(1786,-7114)
\put(3391,-7006){\makebox(0,0)[lb]{\smash{{\SetFigFont{10}{12.0}{\rmdefault}{\mddefault}{\updefault}{\color[rgb]{0,0,0}$x_1$}%
}}}}
\put(2101,-4786){\makebox(0,0)[lb]{\smash{{\SetFigFont{10}{12.0}{\rmdefault}{\mddefault}{\updefault}{\color[rgb]{0,0,0}$\tau_1$}%
}}}}
\put(7216,-2116){\makebox(0,0)[lb]{\smash{{\SetFigFont{10}{12.0}{\rmdefault}{\mddefault}{\updefault}{\color[rgb]{0,0,0}$\gamma(t_2)$}%
}}}}
\put(7276,-7006){\makebox(0,0)[lb]{\smash{{\SetFigFont{10}{12.0}{\rmdefault}{\mddefault}{\updefault}{\color[rgb]{0,0,0}$x_2$}%
}}}}
\put(8566,-4741){\makebox(0,0)[lb]{\smash{{\SetFigFont{10}{12.0}{\rmdefault}{\mddefault}{\updefault}{\color[rgb]{0,0,0}$\tau_2$}%
}}}}
\put(3346,-2116){\makebox(0,0)[lb]{\smash{{\SetFigFont{10}{12.0}{\rmdefault}{\mddefault}{\updefault}{\color[rgb]{0,0,0}$\gamma(t_1)$}%
}}}}
\put(1786,-3391){\makebox(0,0)[lb]{\smash{{\SetFigFont{10}{12.0}{\rmdefault}{\mddefault}{\updefault}{\color[rgb]{0,0,0}$\gamma'(t_1)$}%
}}}}
\put(8506,-3421){\makebox(0,0)[lb]{\smash{{\SetFigFont{10}{12.0}{\rmdefault}{\mddefault}{\updefault}{\color[rgb]{0,0,0}$\gamma'(t_2)$}%
}}}}
\end{picture}%

%% file: crossings.pstex_t
\begin{picture}(0,0)%
\includegraphics{crossings.pstex}%
\end{picture}%
\setlength{\unitlength}{1184sp}%
\begingroup\makeatletter\ifx\SetFigFont\undefined%
\gdef\SetFigFont#1#2#3#4#5{%
  \reset@font\fontsize{#1}{#2pt}%
  \fontfamily{#3}\fontseries{#4}\fontshape{#5}%
  \selectfont}%
\fi\endgroup%
\begin{picture}(13073,2461)(-1653,-5198)
\put(-1653,-4709){\makebox(0,0)[lb]{\smash{{\SetFigFont{7}{8.4}{\rmdefault}{\mddefault}{\updefault}{\color[rgb]{0,0,0}$x_0$}%
}}}}
\put(-213,-4844){\makebox(0,0)[lb]{\smash{{\SetFigFont{7}{8.4}{\rmdefault}{\mddefault}{\updefault}{\color[rgb]{0,0,0}$x$}%
}}}}
\put(9672,-3419){\makebox(0,0)[lb]{\smash{{\SetFigFont{7}{8.4}{\rmdefault}{\mddefault}{\updefault}{\color[rgb]{0,0,0}$y$}%
}}}}
\put(10467,-3404){\makebox(0,0)[lb]{\smash{{\SetFigFont{7}{8.4}{\rmdefault}{\mddefault}{\updefault}{\color[rgb]{0,0,0}$x_1$}%
}}}}
\end{picture}%

%% file: rectangle.pstex_t
\begin{picture}(0,0)%
\includegraphics{rectangle.pstex}%
\end{picture}%
\setlength{\unitlength}{1184sp}%
\begingroup\makeatletter\ifx\SetFigFont\undefined%
\gdef\SetFigFont#1#2#3#4#5{%
  \reset@font\fontsize{#1}{#2pt}%
  \fontfamily{#3}\fontseries{#4}\fontshape{#5}%
  \selectfont}%
\fi\endgroup%
\begin{picture}(6024,7675)(2989,-7195)
\put(5784,-7060){\makebox(0,0)[lb]{\smash{{\SetFigFont{9}{10.8}{\rmdefault}{\mddefault}{\updefault}{\color[rgb]{0,0,0}$p$}%
}}}}
\put(5802,120){\makebox(0,0)[lb]{\smash{{\SetFigFont{9}{10.8}{\rmdefault}{\mddefault}{\updefault}{\color[rgb]{0,0,0}$q$}%
}}}}
\put(6113,-1092){\makebox(0,0)[lb]{\smash{{\SetFigFont{9}{10.8}{\rmdefault}{\mddefault}{\updefault}{\color[rgb]{0,0,0}$\beta$}%
}}}}
\put(6282,-5048){\makebox(0,0)[lb]{\smash{{\SetFigFont{9}{10.8}{\rmdefault}{\mddefault}{\updefault}{\color[rgb]{0,0,0}$\beta'$}%
}}}}
\end{picture}%

%% file: stretch_new.bbl
\begin{thebibliography}{10}

\bibitem{AGW-m}
M.~Akian, S.~Gaubert, and C.~Walsh.
\newblock The max-plus {M}artin boundary.
\newblock {\em Doc. Math.}, 14:195--240, 2009.

\bibitem{ballmann_lectures}
W.~Ballmann.
\newblock {\em Lectures on spaces of nonpositive curvature}, volume~25 of {\em
  DMV Seminar}.
\newblock Birkh\"auser Verlag, Basel, 1995.
\newblock With an appendix by Misha Brin.

\bibitem{bonahon_currents}
F.~Bonahon.
\newblock The geometry of {T}eichm\"uller space via geodesic currents.
\newblock {\em Invent. Math.}, 92(1):139--162, 1988.

\bibitem{burger_mozes_commensurators}
M.~Burger and S.~Mozes.
\newblock {${\rm CAT}$}(-{$1$})-spaces, divergence groups and their
  commensurators.
\newblock {\em J. Amer. Math. Soc.}, 9(1):57--93, 1996.

\bibitem{choi_rafi_comparison}
Y.-E. Choi and K.~Rafi.
\newblock Comparison between {T}eichm\"uller and {L}ipschitz metrics.
\newblock {\em J. Lond. Math. Soc. (2)}, 76(3):739--756, 2007.

\bibitem{diaz_series_lines}
R.~D{\'{\i}}az and C.~Series.
\newblock Limit points of lines of minima in {T}hurston's boundary of
  {T}eichm\"uller space.
\newblock {\em Algebr. Geom. Topol.}, 3:207--234 (electronic), 2003.

\bibitem{earle_kra_isometries}
C.~J. Earle and I.~Kra.
\newblock On isometries between {T}eichm\"uller spaces.
\newblock {\em Duke Math. J.}, 41:583--591, 1974.

\bibitem{fathi_laudenbach_poenaru}
A.~Fathi, F.~Laudenbach, and V.~Po\'enaru.
\newblock {\em Travaux de {T}hurston sur les surfaces}, volume~66 of {\em
  Ast\'erisque}.
\newblock Soci\'et\'e Math\'ematique de France, Paris, 1979.
\newblock S{\'e}minaire Orsay, With an English summary.

\bibitem{gromov:hyperbolicmanifolds}
M.~Gromov.
\newblock Hyperbolic manifolds, groups and actions.
\newblock In {\em Riemann surfaces and related topics: Proceedings of the 1978
  Stony Brook Conference (State Univ. New York, Stony Brook, N.Y., 1978)},
  volume~97 of {\em Ann. of Math. Stud.}, pages 183--213, Princeton, N.J.,
  1981. Princeton Univ. Press.

\bibitem{ivanov_automorphisms}
N.~V. Ivanov.
\newblock Automorphism of complexes of curves and of {T}eichm\"uller spaces.
\newblock {\em Internat. Math. Res. Notices}, (14):651--666, 1997.

\bibitem{ivanov_isometries}
N.~V. Ivanov.
\newblock Isometries of {T}eichm\"uller spaces from the point of view of
  {M}ostow rigidity.
\newblock In {\em Topology, ergodic theory, real algebraic geometry}, volume
  202 of {\em Amer. Math. Soc. Transl. Ser. 2}, pages 131--149. Amer. Math.
  Soc., Providence, RI, 2001.

\bibitem{karlsson_ledrappier_laws}
A.~Karlsson and F.~Ledrappier.
\newblock On laws of large numbers for random walks.
\newblock {\em Ann. Probab.}, 34(5):1693--1706, 2006.

\bibitem{korkmaz_automorphisms}
M.~Korkmaz.
\newblock Automorphisms of complexes of curves on punctured spheres and on
  punctured tori.
\newblock {\em Topology Appl.}, 95(2):85--111, 1999.

\bibitem{walsh_lemmens_polyhedral}
B.~Lemmens and C.~Walsh.
\newblock Isometries of polyhedral {H}ilbert geometries.
\newblock {\em J. Topol. Anal.}, 3(2):213--241, 2011.

\bibitem{luo_automorphisms}
F.~Luo.
\newblock Automorphisms of the complex of curves.
\newblock {\em Topology}, 39(2):283--298, 2000.

\bibitem{papadopoulos_extension}
A.~Papadopoulos.
\newblock On {T}hurston's boundary of {T}eichm\"uller space and the extension
  of earthquakes.
\newblock {\em Topology Appl.}, 41(3):147--177, 1991.

\bibitem{papadopoulos_theret_handbook}
A.~Papadopoulos and G.~Th{\'e}ret.
\newblock On {T}eichm\"uller's metric and {T}hurston's asymmetric metric on
  {T}eichm\"uller space.
\newblock In {\em Handbook of {T}eichm\"uller theory. {V}ol. {I}}, volume~11 of
  {\em IRMA Lect. Math. Theor. Phys.}, pages 111--204. Eur. Math. Soc.,
  Z\"urich, 2007.

\bibitem{papadopoulos_theret_topology}
A.~Papadopoulos and G.~Th{\'e}ret.
\newblock On the topology defined by {T}hurston's asymmetric metric.
\newblock {\em Math. Proc. Cambridge Philos. Soc.}, 142(3):487--496, 2007.

\bibitem{papadopoulos_theret_polyhedral}
A.~Papadopoulos and G.~Th{\'e}ret.
\newblock Shift coordinates, stretch lines and polyhedral structures for
  {T}eichm\"uller space.
\newblock {\em Monatsh. Math.}, 153(4):309--346, 2008.

\bibitem{patterson_distinct}
D.~B. Patterson.
\newblock The {T}eichm\"uller spaces are distinct.
\newblock {\em Proc. Amer. Math. Soc.}, 35:179--182, 1972.

\bibitem{rieffel_group}
M.~A. Rieffel.
\newblock Group {$C\sp *$}-algebras as compact quantum metric spaces.
\newblock {\em Doc. Math.}, 7:605--651 (electronic), 2002.

\bibitem{royden_isometries}
H.~L. Royden.
\newblock Automorphisms and isometries of {T}eichm\"uller space.
\newblock In {\em Advances in the {T}heory of {R}iemann {S}urfaces ({P}roc.
  {C}onf., {S}tony {B}rook, {N}.{Y}., 1969)}, pages 369--383. Ann. of Math.
  Studies, No. 66. Princeton Univ. Press, Princeton, N.J., 1971.

\bibitem{theret_thesis}
G.~Th{\'e}ret.
\newblock {\em \`{A} propos de la m\'etrique asym\'etrique de {T}hurston sur
  l'espace de {T}eichm\"uller d'une surface}.
\newblock Pr\'epublication de l'Institut de Recherche Math\'ematique Avanc\'ee
  [Prepublication of the Institute of Advanced Mathematical Research], 2005/8.
  Universit\'e Louis Pasteur. Institut de Recherche Math\'ematique Avanc\'ee,
  Strasbourg, 2005.
\newblock Th{\`e}se, Universit{\'e} Louis Pasteur (Strasbourg I), Strasbourg,
  2005.

\bibitem{theret_negative}
G.~Th{\'e}ret.
\newblock On the negative convergence of {T}hurston's stretch lines towards the
  boundary of {T}eichm\"uller space.
\newblock {\em Ann. Acad. Sci. Fenn. Math.}, 32(2):381--408, 2007.

\bibitem{theret_elementary}
G.~Th{\'e}ret.
\newblock On elementary antistretch lines.
\newblock {\em Geom. Dedicata}, 136:79--93, 2008.

\bibitem{theret_divergence}
G.~Th{\'e}ret.
\newblock Divergence et parall\'elisme des rayons d'\'etirement cylindriques.
\newblock {\em Algebr. Geom. Topol.}, 10(4):2451--2468, 2010.

\bibitem{thurston_minimal}
W.~Thurston.
\newblock Minimal stretch maps between hyperbolic surfaces.
\newblock preprint, arXiv:math GT/9801039, 1986.

\bibitem{walsh_minimum}
C.~Walsh.
\newblock Minimum representing measures in idempotent analysis.
\newblock In {\em Tropical and idempotent mathematics}, volume 495 of {\em
  Contemp. Math.}, pages 367--382. Amer. Math. Soc., Providence, RI, 2009.

\end{thebibliography}
